\documentclass[a4paper]{article}

\usepackage[english]{babel}
\usepackage[utf8x]{inputenc}
\usepackage[T1]{fontenc}

\usepackage[a4paper,top=3cm,bottom=2cm,left=3cm,right=3cm,marginparwidth=1.75cm]{geometry}

\usepackage{amsmath}
\usepackage{graphicx}
\usepackage[colorinlistoftodos]{todonotes}
\usepackage[colorlinks=true, allcolors=blue]{hyperref}
\usepackage{latexsym}
\usepackage{amsmath,amsthm}
\usepackage{amsfonts}
\usepackage{amsgen}
\usepackage{amssymb}
\usepackage{amstext}
\usepackage{amssymb}
\usepackage{dsfont}

\newtheorem{theorem}{Theorem}[section]
\newtheorem{lemma}[theorem]{Lemma}
\newtheorem{definition}{Definition}[section]
\newtheorem{remark}[theorem]{Remark}
\newtheorem{example}[theorem]{Example}
\newtheorem{corollary}[theorem]{Corollary}

\DeclareMathOperator{\dist}{dist}
\DeclareMathOperator{\diam}{diam}
\DeclareMathOperator{\id}{id}

\title{Unifying topological entropy notions for piecewise continuous maps on the interval}
\author{A.E. Calderón\footnote{Escuela de Ingeniería, Facultad de Ingeniería y Empresa -- Universidad Católica Silva Henríquez, Santiago, Chile. {\tt acalderonc@ucsh.cl}} \ and E. Villar-Sepúlveda\footnote{Department of Engineering Mathematics -- University of Bristol, Bristol, England. {\tt edgardo.villar-sepulveda@bristol.ac.uk}}}

\begin{document}
\maketitle

\begin{abstract}
\noindent We generalize the definition of topological entropy given by Adler, Konheim, and McAndrew (AKM) for piecewise continuous self-maps defined on a compact interval (pc-maps). For this notion of entropy, we prove that the properties of the AKM-entropy in the compact-continuous setting get naturally extended, including that it can be estimated using Bowen's formula disregarding the metric used to compute it. Additionally, for piecewise strictly monotonic pc-maps, we prove that the Misiurewicz-Szlenk formula based on the asymptotic behavior of the number of monotony pieces for iterations of the map coincides with our notion of topological entropy. In this way, we unify the notions of entropy commonly used for piecewise continuous interval maps, proving that they all coincide under compactness assumptions.\\

\end{abstract}

\section{Introduction}\label{Introduction}
    Rufus Bowen was a pioneer in deepening and giving importance to the concept of topological entropy, which was introduced in \cite{AKM65} by Adler, Konheim, and McAndrew (AKM). This concept was originally defined using open covers to study the `complexity' of continuous self-maps on a compact topological space. A reformulation of the AKM-entropy based on the dispersion of orbits was given by Bowen and Dinaburg (BD) \cite{Bo70,D70}. However, that definition requires the topological space to be metrizable. Fortunately, both definitions yield the same number when the metric space is compact. Compactness guarantees that the BD-entropy does not depend on the metric used to calculate it.\\[-2ex]
    
    On the other hand, when the metric space is not compact, the BD-entropy gets metric-dependent. It is well-known that if it is computed using two uniformly equivalent metrics, the results are the same (for more details, see \cite[pages 168-176]{Wa00}). Nonetheless, these can change if the metrics used are equivalent, but not uniformly equivalent (an example of this can be found in \cite[page 171]{Wa00}). Another problem with this definition in the non-compact case is that, in general, the BD-entropy does not seem to measure the complexity of the map's dynamics. For example, consider the uniformly continuous map $T: \mathbb R \to \mathbb R$ given by $T(x) = 2 x$, where $\mathbb R$ is endowed with the Euclidean metric. The dynamics of $T$ is really simple: $0$ is the unique fixed point and, for every $x \neq 0$, $|T^n(x)| \to \infty$ as $n \to \infty$. Nevertheless, for this map, $h_\text{BD}(T) \geq \log 2$ (see \cite{Bo71}), which breaks the premise: ``positive entropy implies that the map has a complicated dynamic behavior''. Unfortunately, this problem is not solved by defining the entropy of $T$ only considering finite open covers of the non-compact set (see \cite[page 241]{Ho75}).\\[-2ex]
    
    To address this issue, other definitions of entropy have been considered to obtain good properties when the space is not compact or the dynamical system is not continuous. In particular, two notions of topological entropy --based on BD-entropy-- are provided in recent papers \cite{CR05,RBK18}. These entropies are referred to as {\it continuity entropies} because they are obtained by restricting the system to different compact subsets where the map is continuous and invariant. Then, one has to take the supremum over all these restrictions (see \cite{RBK18}). These notions are independent of the metric and seem to solve the non-compactness problem. Such definitions preserve some good properties of the BD-entropy and serve as a good measure of the complexity of the system's dynamic behavior in the non-compact context; nevertheless, the property associated with topological semiconjugation is generally not valid. Note that if the system is piecewise continuous, it can always be restricted to the subspace where the system and its iterates are continuous, which is usually non-empty and non-compact.
    Nonetheless, there are cases where complicated dynamics might not be taken into account by continuity entropies in this context. For instance, as shown in \cite{CCG20}, the a\-symp\-to\-tic dynamics of piecewise contracting interval maps may be supported on minimal Cantor sets, which always intersect some jump-discontinuities of the system. The good pro\-per\-ties of the continuity entropies strongly depend on the continuity of the map restricted to compact and invariant subsets, but every piecewise contracting interval map res\-tric\-ted to an attracting Cantor set --if it exists-- is discontinuous, not to mention that these Cantor sets are not exactly invariant (since these appear regardless of how the map is defined over discontinuities). These sets satisfy what is known as {\em pseudo-invariance}, which was introduced in \cite{CCG20} and allows working with a concept of invariance that does not depend on how the map is defined over its finite set of discontinuity points (a rigorous definition of pseudo-invariance is provided in Section \ref{sec2}). In short, the problem with the concepts of continuity entropies is that there are important properties that are not valid in the discontinuous setting, and --sometimes-- they fail to capture some complicated dynamics.\\[-2ex]
    
    Other formulations of entropy have been used in a one-dimensional setting. In \cite{H79,Y81}, symbolic dynamics is used to define the entropy of piecewise strictly monotonic pc-maps (i.e. the maps that are continuous on a compact interval and strictly monotonic over a \underline{finite} number of intervals). This ``symbolic entropy'' coincides with the AKM-entropy when the system is continuous (see \cite{Y81}) or when the set of pre-images of turning points is dense on the whole interval (see \cite{H79}). For the same class of systems (i.e. piecewise strictly monotonic pc-maps), we have the following formula due to Misiurewicz and Szlenk (see \cite{MS80}): If $T$ is a continuous piecewise strictly monotonic interval map, then
    \begin{equation}
        h_\text{AKM}(T) = \lim_{n \to \infty} \frac{1}{n} \log c_n, \label{CN}
    \end{equation}
    where $c_n \geq 1$ is the smallest number of intervals over which $T^n$ is monotonous. However, if $T$ is not continuous, it is natural to consider the value \eqref{CN} as the definition of topological entropy (see \cite{AM14,BKP02,M03,GN22}), which is also referred to as the {\em singular entropy} in the discontinuous context, differentiating it from the topological entropy defined using Bowen's formula (see, for example, \cite{R17}), which has also been used to provide an idea of ``dynamic complexity'' associated with piecewise continuous dynamical systems (see \cite{CGMU16,CV23}). Considering this, the singular entropy of a piecewise strictly monotonic pc-map can be computed by counting permutations exhibited by periodic orbits (see \cite{M03}) or by initial segments of orbits (see \cite{BKP02}). However, in \cite{M03} it is shown that these approximations via permutations can fail if $T$ is not piecewise monotonic. Also, note that defining the entropy of $T$ via symbolic dynamics or formula \eqref{CN} makes it difficult to obtain good properties for such a concept. Last, in \cite{C04}, the concept of topological entropy for arbitrary self-maps defined on a compact interval is defined using Bowen's formula. Whilst it is possible to obtain good properties from this notion of entropy, the dependence on the metric is explicit.\\[-2ex]
    
    Our goal in this article is to construct a metric-independent concept of topological entropy that does not depend on how the map is defined over its discontinuity points whilst keeping all of the good properties that the AKM-entropy has in the compact-continuous setting. Particularly, we want this concept of entropy to capture all complicated dynamics, unlike what happens with the continuity entropies defined in \cite{CR05,RBK18}. We will achieve this by generalizing the classical topological construction of the AKM-entropy in \cite{AKM65}, ensuring that our definition is independent of the metric that generates the topology of the space. \\[-1ex]

    \noindent {\bf Outline of the article:} In Section \ref{sec2}, we establish the main results that are proven in this article. Sections \ref{sec3} and \ref{sec4} are dedicated to the construction of the concept of topological entropy for pc-maps with a finite number of critical points. Specifically, on the one hand, in Section \ref{sec3} we introduce and prove some basic results on preliminary concepts such as collections, open covers, and minimal cardinality, all adapted to our context. On the other hand, in Section \ref{sec4} we define our notion of topological entropy for piecewise continuous interval maps and prove that it is well-defined. In Section \ref{sec6}, we prove some classic properties that this notion of entropy satisfies (see Theorem \ref{ENTROPY}). In Section \ref{sec5}, we prove that it is possible to use Bowen's formula to compute the topological entropy, concluding that this concept does not depend on the metric used as long as the space is compact (see Theorem \ref{BOWEN.MAIN}). In Section \ref{sec7}, we prove that the topological entropy can also be estimated using Misiurewicz-Szlenk formula (see Theorem \ref{main2}). As a corollary of this, we prove that every injective pc-map has zero topological entropy. Furthermore, we present a result on how discontinuities and/or critical points influence the value of the topological entropy for systems that map each continuity piece onto the entire interval (see Theorem \ref{main3}). Finally, Section \ref{sec8} is devoted to showing some examples and applications of the theory.

\section{Preliminary definitions and main results}\label{sec2}
    Let $X \subset \mathbb R$ be a compact interval with a non-empty interior. A map $f: X \to X$ is said to be a {\em piecewise continuous interval map} (in short, {\em pc-map}) if there exists a finite collection of $N \geq 1$ open subsets $X_1, X_2, \ldots, X_N\subset X$, called {\em continuity pieces} of $f$, such that 
    \begin{enumerate}
        \item $X_{i_1} \cap X_{i_2} = \emptyset \,$ for all $i_1, i_2 \in \{1, \ldots, N\}$ such that $i_1\neq i_2\,$;
        \item $X = \overline{X_1}\cup \overline{X_2} \cup \ldots \cup \overline{X_N}\,$;
        \item $f|_{X_i}$ is a continuous map for each $i \in \{1, \ldots, N\}$.
    \end{enumerate}
    \noindent Furthermore, we assume that $f$ has a finite number of critical points (i.e., points where $f$ is not locally invertible or where $f$ is discontinuous) within each continuity piece. The technical reasons behind this assumption are explained in item {\it (a)} of Remark \ref{CRITICAL}. Here, the topology considered on $X$ is the one induced by the usual topology of $\mathbb R$. Also, we denote the topological boundary operator within $X$ as $\partial$. Specifically, if $A \subset X$, then $\partial A$ represents those points in $X$ for which every open neighborhood intersects $A$ and $X \setminus A$.\\[-2ex]
    
    For a pc-map $f: X \to X$ we will consider its continuity pieces to be sorted; i.e., for all $i, j \in\{1, \ldots, N\}$ with $i < j$, we assume that $x < y$ for every $(x, y) \in X_i \times X_j$. Let $\Delta_f$ be the {\em set of boundaries of the continuity pieces} of $f$; that is,
    \[
        \Delta_f := \bigcup_{i = 1}^N \partial X_i.
    \]
    For convenience of notation, we assume that $f$ is continuous at the endpoints of $X$, which implies that the endpoints of $X$ do not belong to $\Delta_f$  (see item {\it (b)} of Remark \ref{CRITICAL}). Note that the continuity pieces $X_1$ and $X_N$ are semi-open.  \\[-2ex]

    Now, for every $m \geq 0$, we define the set
    \[
        \Delta_f^m = (\Delta_f)^m:= \bigcup_{j = 0}^{m - 1} f^{- j}(\Delta_f)\subset X,
    \]
    where, by convention, we consider that $\Delta_f^0 = \emptyset$. As $f$ has a finite number of critical points, for every $m \geq 0$ the set $\Delta_f^m$ is finite and, consequently, a closed set. Furthermore, for every $k \geq 1$, we define the set of boundaries of the continuity pieces of the map $f^k$, denoted by $\Delta_{f^k}$, as the collection induced by $\Delta_f$; that is, $\Delta_{f^k}:= \Delta_f^k$. \\[-2ex]
    
    We highlight that $\Delta_f$ can contain removable or jump discontinuities of $f$. As usual, the orbit of $x\in X$ (by $f$) is defined as the countable set $\{f^n(x)\}_{n \geq 0}$. Note that this notion depends on how $f$ is defined on $\Delta_f$. For this reason, it is convenient to work on the set $\widetilde X_f$, which is defined as the collection of points in $X$ whose orbit does not intersect $\Delta_f$; that is,
    \[
        \widetilde X_f:= \bigcap_{n \geq 0} f^{- n}(X \setminus \Delta_f).
    \]
    Note that $\widetilde X_f = X$ if and only if $\Delta_f = \emptyset$. For the iterates of $f$, we have the following result: 

    \begin{lemma}\label{Xtilde.F}
        For every pc-map $f: X \to X$ and every integer $k \geq 1$, $\widetilde X_f = \widetilde X_{f^k}$.
    \end{lemma}
    \begin{proof}
        For every integer $k \geq 1$, it is clear that $\widetilde X_f \subset \widetilde X_{f^k}$. Reciprocally, if $x \in \widetilde X_{f^k}$, we have that $x \in f^{- j k}(X \setminus \Delta_{f^k})$ for all $j \geq 0$. This implies that
        \[
            f^{j k}(x) \in X \setminus \Delta_{f^k} \quad \forall \, j \geq 0.
        \]
        Thus, we deduce that $f^{j k + \ell}(x) \notin \Delta_f$ for every $j \geq 0$ and $\ell \in \{0, \ldots, k - 1\}$. Therefore, $f^i(x) \notin \Delta_f$ for every $i \geq 0$. This argument shows that $\widetilde X_f = \widetilde X_{f^k}$.
    \end{proof}
    
    Since each pc-map has a finite number of critical points, it is not difficult to see that $\widetilde X_f$ is non-empty and dense in $X$.\\[-2ex]
    
    Recall that a set $A\subset X$ is $f$-{\em invariant} if $f(A) \subset A$. Note that $\widetilde X_f$ is an $f$-invariant --usually non-compact-- set, and the restriction $f|_{\widetilde X_f}: \widetilde X_f \to \widetilde X_f$ is a continuous self-map. As it happens with the concept of orbit, the notion of invariance depends on the definition of $f$ over $\Delta_f$. There are 2 ways to solve this: ({\it 1}) Require $A \subset X$ to be an $f$-invariant subset such that $A \cap \widetilde X_f \neq \emptyset$ (note that, in this case, $A \cap \widetilde X_f$ is $f$-invariant as well); ({\it 2}) Using a new type of invariance called {\em pseudo-invariance} (defined below), which does not depend on how $f$ is defined on $\Delta_f$.
    
    We say that a set $A \subset X$ is $f$-{\em pseudo-invariant} if for every $x_0 \in A$, the limit of $f(x)$ as $x \to x_0$ either from above or below, exists and belongs to $A$. In particular, note that if $A \subset X$ is such that $A \cap \Delta_f = \emptyset$, then $A$ is $f$-invariant if and only if $A$ is $f$-pseudo-invariant. Also, if $A$ is an $f$-pseudo-invariant set, then $f(A \setminus \Delta_f) \subset A$.
    Under some conditions, it can be guaranteed that any $f$-pseudo-invariant set intersects $\widetilde X_f$: if $f$ is a pc-map such that for every $c \in \Delta_f$, the limit of $f(x)$ as $x \to c$, either from above or below, exists and belongs to $\widetilde X_f$, then every $f$-pseudo-invariant subset contained in $X$ intersects $\widetilde X_f$ (this result follows directly from the definition). However, it will be convenient to require the corresponding $f$-(pseudo-)invariant subset intersects $\widetilde X_f$ in each statement.
    
    Although invariant and $f$-pseudo-invariant sets generate invariant sets when we intersect them with $\widetilde X_f$ (and the intersection is non-empty), both approaches are interesting because invariance is neither a necessary nor a sufficient condition for $f$-pseudo-invariance \footnote{For example, the $f$-pseudo-invariant Cantor sets that can appear in the asymptotic dynamics of piecewise contracting interval maps are not invariant, in general (see \cite{CCG20}). On the other hand, any invariant subset containing an isolated discontinuity $c \in \Delta_f$ such that $f(c) = c$ is not $f$-pseudo-invariant.}. 
    
    \begin{definition} \label{CONSTANZA}
        Let $X$ and $Y$ be two compact intervals and let $f: X \to X$ and $g: Y \to Y$ be two pc-maps defined on $X$ and $Y$, respectively. We say that $f$ is {\em topologically semi-conjugate} (respectively, {\em conjugate}) to $g$ if there exists a continuous surjection (respectively, homeomorphism) $\varphi: X \to Y$ such that $\varphi^{-1}(\Delta_g)=\Delta_f$ and $\varphi \circ f = g \circ \varphi$ on $X \setminus \Delta_f$.
    \end{definition}

    
    \begin{remark}\label{REMARK.CONJUGATE}
        \begin{enumerate}
            \item In Definition \ref{CONSTANZA}, the condition $\varphi^{- 1}(\Delta_g) = \Delta_f$ is not needed when $\Delta_f$ and $\Delta_g$ only comprise jump discontinuity points of $f$ and $g$, respectively. We use it only to be general without imposing that assumption.
            \item If $f: X \to X$ and $g: Y \to Y$ are two pc-maps such that $f$ is topologically semi-conjugate to $g$ by a continuous surjection $\varphi: X \to Y$, then $\# \Delta_f \geq \# \Delta_g$, $\varphi(\Delta_f) = \Delta_g$ and $\varphi(\widetilde X_f) = \widetilde Y_g$.
        \end{enumerate}
    \end{remark}

    We recall that the {\it diameter} of a set $E \subset X$, denoted by $\diam(E)$, is defined as the supremum of the distances between all pairs of points that belong to $E$ (by convention, we say that $\diam(\emptyset) = 0$). Thus, the {\em diameter of a collection} $\mathcal C$ of subsets of $X$ is given by
    \[
        \diam(\mathcal C) := \sup_{C \in \mathcal C} \diam(C).
    \]
    Note that $\diam(\mathcal C) \leq \diam(X) < \infty$.

    To preserve the usual notation, the topological entropy that we will construct in Sections \ref{sec3} and \ref{sec4} will be denoted by $h_{top}$. Now, we establish the main results of this paper, which are divided into five theorems.\\[-2ex]

    \begin{theorem} \label{ENTROPY}
        Let $X$ and $Y$ be two compact intervals and let $f: X \to X$ and $g: Y \to Y$ be two pc-maps. Then, the following properties hold:
        \begin{enumerate}
            \item[(1)] If $f$ is continuous, then $h_{top}(f)$ coincides with the AKM-entropy of $f$.
            \item[(2)] If $A$ is a non-empty, compact and $f$-(pseudo-)invariant subset of $X$ such that $A\cap\widetilde X_f\neq\emptyset$, then $h_{top}(f|_A) \leq h_{top}(f)$.
            \item[(3)] If $f$ is topologically semi-conjugate (respectively, conjugate) to $g$, then $h_{top}(f)\geq h_{top}(g)$ (respectively, $h_{top}(f) = h_{top}(g)$).
            \item[(4)] $h_{top}(f^k) = k \, h_{top}(f)$ for every integer $k \geq 0$ (by convention, $f^0$ is the identity map).
            \item[(5)] If $X = A_1 \cup \ldots \cup A_p$, where $A_i$ is a compact $f$-(pseudo-)invariant subset of $X$ such that $A_i \cap \widetilde X_f \neq \emptyset$ for each $1\leq i\leq p$, then
            \[
                h_{top}(f) = \max_{1 \leq i \leq p} \! \big\{h_{top}(f|_{A_i})\big\}.
            \]
            \item[(6)] If $h_{top}(f) < \infty$ (respectively, $h_{top}(f) = \infty$) and $\{\mathcal C_k\}_{k \geq 1}$ is a sequence of open covers of $X$ such that $\diam(\mathcal C_k)\to 0$ as $k \to \infty$, then $\lim_{k \to \infty} h_{top}(f, \mathcal C_k) = h_{top}(f)$ (respectively, $\lim_{k \to \infty} h_{top}(f, \mathcal C_k) = \infty$).
        \end{enumerate}
    \end{theorem}
    
    Moreover, the following result will enable us to compute the entropy using the classic Bowen's formula while retaining the same good properties of the compact-continuous setting.

    \begin{theorem}[{[Bowen's formula]}]\label{BOWEN.MAIN}
        Let $f: X \to X$ be a pc-map. If $A$ is a non-empty, compact and $f$-(pseudo-)invariant subset of $X$ such that $A\cap\widetilde X_f\neq\emptyset$, then
        \begin{equation}
            h_{top}(f|_A) = \lim_{\epsilon \to 0} \limsup_{n \to \infty} \frac{\log s_n(A; f, \epsilon)}{n} = \lim_{\epsilon \to 0} \limsup_{n \to \infty} \frac{\log r_n(A; f, \epsilon)}{n},\label{LIMITS.BOWEN.MAIN}
        \end{equation}
        where $s_n(A; f, \epsilon)$ (respectively, $r_n(A; f, \epsilon)$) is the maximal cardinality of $(n, \epsilon)$-separated sets (respectively, minimal cardinality of $(n, \epsilon)$-spanning sets) of $A \cap \widetilde X_f$. Moreover, for the calculation using Bowen's formula, the value of $h_{top}(f|_A)$ does not depend on the metric that generates the induced topology on $A\cap\widetilde X_f$.
    \end{theorem}

    \begin{remark}
        In \eqref{LIMITS.BOWEN.MAIN}, the values 
        \begin{equation*}
            \lim_{\epsilon \to 0} \limsup_{n \to \infty} \frac{\log s_n(A; f, \epsilon)}{n} \quad\text{ and }\quad \lim_{\epsilon \to 0} \limsup_{n \to \infty} \frac{\log r_n(A; f, \epsilon)}{n}
        \end{equation*}
        refer to the limits obtained via the classic Bowen's formula (see \cite[page 58]{R17}) for the map $f$ restricted to the invariant subset $A \cap \widetilde X_f$.
    \end{remark}

    \vspace*{0ex}
    Next, for each $n \geq 1$, we denote by $c_n \geq 1$ the smallest number of intervals on which $f^n$ is monotonic and {\em essentially} continuous; that is, 
    \begin{align}
        c_n:= \min\!\big\{\#\mathcal P \ : \ \mathcal P \in \Phi(f^n)\big\}, \label{Emily}
    \end{align}
    where $\Phi(f^n)$ denotes the collection of partitions $\mathcal P$ comprising intervals of $X$ with non-empty interior such that for all $P\in\mathcal P$, the restriction of $f^n$ to the interior of $P$ is monotonic and {\it continuously extendable} (i.e., every point within the interior of $P$ is either a point where $f^n$ is continuous or a removable discontinuity of $f^n$). Note that adding removable discontinuities to the map $f$ increases its number of continuity pieces, but does not modify $c_n$.\\[-2ex]

    The following result relates the entropy with formula \eqref{CN} adapted to this context.
    
    \begin{theorem}[{[Misiurewicz-Szlenk formula]}] \label{main2}
        Let $f: X\to X$ be a pc-map. Then
        \[
            h_{top}(f) = \lim_{n \to \infty} \frac{1}{n} \log c_n,
        \]
        where $c_n$ is defined as in \eqref{Emily}.
    \end{theorem}

    In addition, we say that a pc-map is {\it injective} when it is injective over the union of its continuity pieces. Thus, as a consequence of Theorem \ref{main2}, we can state an extension of the following result: ``{\it injective continuous maps defined on a compact interval have zero topological entropy}'' (see \cite[Lemma 8.3.1]{V14}).

    \begin{corollary} \label{GIORNO-GIOVANNA}
        Let $f: X \to X$ an injective pc-map. Then $h_{top}(f) = 0$.
    \end{corollary}
    
    Furthermore, under certain conditions, it is possible to obtain the value of the topological entropy as a function of the minimum number of monotonic pieces of $f$:\\[-1ex]

    \begin{theorem}\label{main3}
        Assume that $f: X\to X$ is a pc-map such that $\overline{f(X_i)} = X$ for each $i \in \{1, \ldots, N\}$, then $h_{top}(f) = \log c_1$, where $c_1$ represents the number of strictly monotonic pieces of $f$.\\
    \end{theorem}

    \begin{remark} \label{CRITICAL}
        \begin{itemize}
            \item[{\it (a)}] As any pc-map $f$ has a finite number of critical points, it follows that the preimages of $\Delta_f$ (by $f$) are also finite sets; that is,  $\#\big(f^{- k}(\Delta_f)\big) < \infty$ for every $k \geq 1$. This condition will be frequently used to ensure, for example, that the covers we construct are open (see Lemma \ref{OPENCOVER}), the iteration of a pc-map remains a pc-map (i.e., for every $k \geq 1$, $f^k$ has a finite number of continuity pieces), and $\widetilde X_f \neq \emptyset$. \\[-2ex]
            \item[{\it (b)}] With the assumptions we have made, such as assuming that the endpoints of $X$ are not boundaries of continuity pieces (i.e., they do not belong to $\Delta_f$), some aspects get simplified; for example, the following equality holds: $N = \#\Delta_f + 1$. However, there is no problem with considering one or both extremes of $X$ as points in $\Delta_f$. Doing this does not generate additional continuity pieces. Nonetheless, this change implies the modification of $X_1$, $X_N$, $\widetilde X_f$, and every definition that depends on $\Delta_f$. In particular, these changes could modify the statements or some of the procedures used to prove the results slightly. In any case, the conclusions of this paper are still valid.
        \end{itemize}
    \end{remark}

\section{Collections, covers and minimal cardinality} \label{sec3}
    In this section, we state the notation we use to define the entropy of pc-maps. Throughout this paper, we set a compact interval $X$ and a pc-map $f: X \to X$. For ease of notation, in all that follows, when there is no place for confusion, we will use $\Delta:= \Delta_f$ and $\widetilde X:= \widetilde X_f$, together with the concept of (pseudo-)invariance when talking about $f$-(pseudo-)invariance, omitting the name of the map.

    We denote by $\mathcal P(X)$ the power set of $X$ and call $\mathcal C \subset \mathcal P(X)$ a {\em collection of subsets} of $X$ (or just a {\em collection} of $X$) when $\emptyset \not \in \mathcal C$. Furthermore, we say that the collection $\mathcal C$ is {\em finite} if $\# \mathcal C < \infty$ (note that $\emptyset$ is a finite collection since $\#\emptyset = 0 < \infty$ and $\emptyset \notin \emptyset$). In particular, we say that the collection $\mathcal C$ is a {\em cover} of a subset $A\subset X$ if $\mathcal C \neq \emptyset$ and $A \subset \bigcup_{C \in \mathcal C} C$. Moreover, we define the {\em $j$-th pre-image of the collection} $\mathcal C$ as $f^{- j}(\mathcal C) := \{f^{- j}(C) \ : \; C \in \mathcal C\} \setminus \{\emptyset\}$, for every integer $j \geq 0$ (we recall that, by convention, $f^0$ is the identity map).
    
    Let $\mathcal C_1, \mathcal C_2, \ldots, \mathcal C_m$ be collections of $X$. We define the $\vee$-product of these collections as
    \begin{equation}
        \bigvee_{j = 1}^m \mathcal C_j := 
        \left\{\bigcap_{j = 1}^m C_j \ \ : \ \quad C_j \in \mathcal C_j \quad \forall \, j \in \{1, \ldots, m\}\right\} \setminus \{\emptyset\}, \label{LOR}
    \end{equation}
    which is a new collection of $X$. Also, if $\mathcal C$ and $\mathcal D$ are collections of $X$, we say that $\mathcal D$ is {\em finer} than $\mathcal C$, and write $\mathcal C \preccurlyeq \mathcal D$, if for every $D\in\mathcal D$ there exists $C\in\mathcal C$ such that $D\subset C$. With this, it is clear that for every pair of collections $\mathcal C$ and $\mathcal D$, $\mathcal C \preccurlyeq \mathcal C \preccurlyeq \mathcal C\vee\mathcal D \preccurlyeq \emptyset$.
    
    \begin{remark}\label{VARIOS}
        Here, we mention some important facts about the previous definitions:
        \begin{enumerate}
            \item[({\it a})] Suppose that $\mathcal C_1, \ldots, \mathcal C_m$ are collections of $X$ and $\mathcal D_1, \ldots, \mathcal D_m$ are pairwise non-empty sub-collections of the previous collections; that is, $\emptyset\neq \mathcal D_j\subset \mathcal C_j$ for every $j\in\{1, \ldots, m\}$. Then, $\mathcal D_1\vee\ldots\vee\mathcal D_m$ is a sub-collection of $\mathcal C_1\vee\ldots\vee\mathcal C_m$.
            \item[({\it b})] Any (finite) cover of $X$ is a (finite) cover of any subset $A\subset X$.
            \item[({\it c})] The finite $\vee$-product of covers of a subset $A\subset X$ is a cover of $A$ as well.
        \end{enumerate}
    \end{remark}
    Let $\mathcal C$ be a collection of $X$. Given $F\subset X$ and $n \geq 1$, we consider the new collections
    \[
        \mathcal C \setminus F := \big\{C\setminus F\ :\ \; C\in\mathcal C\big\}\setminus \{\emptyset\}, \qquad \text{and} \qquad \mathcal C^n_f := \bigvee_{j = 0}^{n - 1} f^{-j}\big(\mathcal C \setminus \Delta\big).
    \]
    When there is no place for confusion, we will often use the notation $\mathcal C^n:= \mathcal C^n_f$, omitting the name of the map. Note that for every $n \geq 0$, $\mathcal C^n = \mathcal C^n \setminus \Delta^n$. Moreover, if $\mathcal C$ and $\mathcal D$ are collections of subsets of $X$ such that $\mathcal C \preccurlyeq \mathcal D$, then $\mathcal C^n \preccurlyeq \mathcal D^n$ for every $n \geq 1$.
    
    \begin{lemma} \label{FINITECOVERS}
        Let $A$ be a subset of $X$. Then,
        \begin{enumerate}
            \item $A \subset f^{- j}(A)$ for all $j\geq0$ if and only if $A$ is $f$-invariant. Moreover, if $A$ is an invariant subset of $X$ such that $A\cap\widetilde X\neq\emptyset$ and $\mathcal C$ is a cover of $A$, then $\mathcal C^n$ is a cover of $A\cap\widetilde X$ for every $n \geq 1$.
            \item If $A$ is a non-empty pseudo-invariant subset of $X$ such that $A\cap\widetilde X\neq\emptyset$, then $A\cap\widetilde X \subset f^{-j}(A\setminus\Delta)$ for all $j\geq0$. Moreover, if $A$ is a pseudo-invariant subset of $X$ and $\mathcal C$ is a cover of $A$, then $\mathcal C^n$ is a cover of $A\cap\widetilde X$ for every $n\geq1$.
        \end{enumerate}
    \end{lemma}
    
    \begin{proof}
        {\it (i)} If $A\subset f^{-j}(A)$ for all $j\geq 0$, then taking $j = 1$, we see that
        \[
            A\subset f^{-1}(A) \quad \Rightarrow \quad f(A) \subset f(f^{-1}(A))\subset A,
        \]
        which shows that $A$ is $f$-invariant. Reciprocally, let $j \geq 0$ and assume that $f(A) \subset A$. Then $f^{- j + 1}(A) \subset f^{-j}(f(A))\subset f^{-j}(A)$, which implies that
        \[
            A = f^0(A) \subset f^{-1}(A)\subset f^{-2}(A) \subset \ldots \subset f^{-j}(A).
        \]
        Now, if $A$ is $f$-invariant and $\mathcal C$ is a cover of $A$, then for every $j\geq 0$,
        \[
            A \cap \widetilde X \subset \left(\bigcup_{C\in\mathcal C} f^{-j}(C)\right) \cap \widetilde X \subset \bigcup_{C \in \mathcal C} f^{-j}(C \setminus \Delta),
        \]
        which proves that $f^{-j}(\mathcal C \setminus \Delta)$ is a cover of $A \cap \widetilde X$ for all $j \geq 0$. Thus, we deduce that $\mathcal C^n$ is a cover of $A \cap \widetilde X$ for every $n \geq 1$.\\[1ex]
        {\it (ii)} Let $j \geq 0$. Since $A \subset X$ is pseudo-invariant and $f^j$ is a continuous map on $A\setminus \Delta^j$, we deduce that $A\setminus \Delta^j$ is $f^j$-invariant and
        \begin{equation}
            A\setminus \Delta^j \subset f^{-j}\big(f^j\left(A \setminus \Delta^j\right)\big) \subset f^{-j}\left(A \setminus \Delta^j\right) \subset f^{-j}(A \setminus \Delta). \label{TECH.DETAIL}
        \end{equation}
        Thus, we deduce that
        \begin{equation}
            A \cap \widetilde X = \left(A \setminus \Delta^j\right) \cap \widetilde X \subset f^{-j}(A \setminus \Delta) \cap \widetilde X \subset f^{-j}(A \setminus \Delta). \label{PSE.INV1}
        \end{equation}
        Finally, note that if $\mathcal C$ is a cover of $A$, the last part of the proof follows from \eqref{PSE.INV1}, using an argument similar to the one used at the end of the proof of part (i).
    \end{proof}
    
    \begin{definition}
        Suppose that $\emptyset \neq A\subset X$ and let $\mathcal C$ be a cover of $A$. We define the {\em minimal cardinality of sub-covers of $A$ with respect to $\mathcal C$}, denoted by $\aleph_A(\mathcal C)$, as
        \[
            \aleph_A(\mathcal C):=\inf\!\left\{m\geq1\ :\ \exists\, \mathcal D \subset\mathcal C\text{ with $\#\mathcal D = m$ such that }A\subset\bigcup_{D\in\mathcal D} D\right\}.
        \]
    \end{definition}
    \noindent Note that $1 \leq \aleph_A(\mathcal C) \leq \#\mathcal C$, where $\#\mathcal C$ might be infinity. Some important properties of the minimal cardinality are established in the following result.
    
    \begin{lemma}\label{NPROPERTIES}
        Let $A$ be a subset of $X$ and let $\mathcal C, \mathcal D$ be covers of $A$ such that $\aleph_A(\mathcal C)$ and $\aleph_A(\mathcal D)$ are finite numbers. Then, the following properties hold:\\[-2ex]
        \begin{enumerate}
            \item[({\it a})] If $\mathcal C\preccurlyeq\mathcal D$, then $\aleph_A(\mathcal C)\leq\aleph_A(\mathcal D)$.
            \item[({\it b})] If $\mathcal C\subset\mathcal D$, then $\aleph_A(\mathcal C)\geq\aleph_A(\mathcal D)$.
            \item[({\it c})] $\aleph_A(\mathcal C\vee\mathcal D)\leq \aleph_A(\mathcal C)\cdot \aleph_A(\mathcal D)$.
            \item[({\it d})] If $A$ is $f$-invariant, then $\aleph_A(f^{-j}(\mathcal C)) \leq \aleph_A(\mathcal C)$ for all $j \geq 0$.
        \end{enumerate}
    \end{lemma}
    \begin{proof}
        Define $m:=\aleph_A(\mathcal C)$, and $n:=\aleph_A(\mathcal D)$. Moreover, let $\mathcal M \subset \mathcal C, \, \mathcal N \subset \mathcal D$ be finite covers of $A$ such that $\# \mathcal M = m$ and $\# \mathcal N = n$. These definitions will be used throughout the whole proof.
        \\[1ex]
        \noindent{({\it a})} If $\mathcal C\preccurlyeq\mathcal D$, then for all $N\in \mathcal N$ there exists $M_{N}\in \mathcal C$ such that $N\subset M_N$. Thus, we deduce that
        \[
            A\subset \bigcup_{N\in\,\mathcal N} N \subset \bigcup_{N\in\,\mathcal N} M_N.
        \]
        Therefore, as $n = \#\mathcal N \geq \#\{M_N\ :\ N\in\mathcal N\}\geq m$, then we conclude ({\it a}).
        \\[1ex]
        \noindent ({\it b}) If $\mathcal C \subset \mathcal D$, then $\mathcal D \preccurlyeq \mathcal C$. Thus, the result follows from part ({\it a}).
        \\[1ex]
        \noindent{({\it c})} The collection $\mathcal M\vee\mathcal N\subset \mathcal C\vee\mathcal D$ is a finite cover of $A$ such that $\aleph_A(\mathcal C\vee\mathcal D)\leq \aleph_A(\mathcal M\vee\mathcal N)\leq mn$, which implies that $\aleph_A(\mathcal C\vee\mathcal D)\leq \aleph_A(\mathcal C)\cdot \aleph_A(\mathcal D)$, concluding the proof of part ({\it c}).
        \\[1ex]
        \noindent{({\it d})} Due to part (i) of Lemma \ref{FINITECOVERS}, since $A$ is $f$-invariant, for $j\geq 0$ we have that
        \[
            A\subset \bigcup_{M\in\,\mathcal M}M \qquad \Rightarrow\qquad A\subset f^{-j}(A)\subset \bigcup_{M\in\,\mathcal M} f^{-j}(M),
        \]
        which lets us deduce that $f^{- j}(\mathcal M)$ is a finite cover of $A$ with no more than $m$ elements. Then,
        \[
            \aleph_A(f^{-j}(\mathcal C))\leq \aleph_A(f^{-j}(\mathcal M))\leq m=\aleph_A(\mathcal M)=\aleph_A(\mathcal C),
        \]
        concluding the proof of part ({\it d}).
    \end{proof}
    
    \begin{remark}\label{REMARK.CAP}
        Under the same assumptions as in Lemma \ref{NPROPERTIES}, we have that
        \[
            A\subset \bigcup_{C \in \mathcal C} C \qquad \Leftrightarrow\qquad A = \bigcup_{C \in \mathcal C}(C \cap A).
        \]
        Thus, we deduce that $\aleph_{A}(\mathcal C \cap A) = \aleph_A(\mathcal C)$, where
        \[
            \mathcal C \cap A := \mathcal C \vee \{A\} = \big\{C \cap A\ :\ C\in \mathcal C\big\} \setminus \{\emptyset\}.
        \]
    \end{remark}
    
    \begin{remark}\label{LORC.NOTOPEN}
        Fix $n \geq 1$ and consider the same hypotheses as in Lemma \ref{NPROPERTIES}. Then, due to part {\it (d)} of said lemma, we can always find a sub-collection of $\mathcal C^n$ that is a finite cover of $A\cap\widetilde X$. Thus, $\aleph_{A \cap \widetilde X}(\mathcal C^n) < \infty$ for every integer $n \geq 1$.
    \end{remark}
    
    Let $A, B$ be non-empty subsets such that $A\subset B\subset X$ and let $\mathcal C$ be a cover of $A$. We say that $\mathcal C$ is a $B$-{\em open cover} of $A$ if every $C\in\mathcal C$ is an open set of $B$ (endowed with the subspace topology). To avoid redundancy, when $A = B$ we will say ``$A$-open cover'' when talking about an ``$A$-open cover of $A$''. Naturally, a {\em finite $B$-open cover of $A$} is a $B$-open cover of $A$ with a finite number of elements.

    
    Assume that $A\subset X$ is a (pseudo-)invariant set such that $A \cap \widetilde X \neq \emptyset$. Let $\mathcal C$ be an $A$-open cover. From Lemma \ref{FINITECOVERS}, we have that $\mathcal C^n$ is a cover of $A \cap \widetilde X$ for every $n \geq 1$. Now, we also prove that every element of $\mathcal C^n$ is $A$-open.
    
    \begin{lemma} \label{OPENCOVER}
        Let $A$ be a (pseudo-)invariant subset of $X$ such that $A \cap \widetilde X \neq \emptyset$ and let $\mathcal C$ be an $A$-open cover. Then, $\mathcal C^n$ is an $A$-open cover of $A \cap \widetilde X$ for every $n \geq 1$.
    \end{lemma}
    \begin{proof}
        We proceed by induction on $n$. If $n = 1$, the result follows. Next, assume that for $k\geq 1$, $\mathcal C^k$ is an $A$-open cover of $A\cap \widetilde X$. Note that
        \begin{align*}
            \mathcal C^k = \left(\bigvee_{j = 0}^{k - 1} f^{-j}(\mathcal C)\right)\setminus \Delta^k.
        \end{align*}
        With this, we deduce that
        \begin{align}
            \mathcal C^{k + 1} &= \mathcal C^k \vee \big(f^{-k}(\mathcal C) \setminus f^{-k}(\Delta)\big) = \mathcal C^k \vee \big(f^{-k}(\mathcal C) \setminus \Delta^{k + 1}\big). \label{CKV}
        \end{align}
        Since $\mathcal C^k$ is an $A$-open cover of $A\cap \widetilde X$, and $f^{-k}(\Delta) \subset \Delta^{k + 1}$, we have that $\mathcal C^{k + 1}$ is an $A$-open cover of $A\cap \widetilde X$, which ends the proof.
    \end{proof}
    
    \begin{remark}\label{COVERS.N}
        Let $\mathcal C,\mathcal D$ be covers of a set $A\subset X$. If $A\cap\widetilde X\neq\emptyset$ and
        \[
            \mathcal C \cap \widetilde X = \mathcal D \cap \widetilde X,
        \]
        then the least number of elements required to cover $A\cap \widetilde X$ with either $\mathcal C$ or $\mathcal D$ is the same. Specifically,
        \[
            \aleph_{A \cap \widetilde X}(\mathcal C) = \aleph_{A\cap\widetilde X}(\mathcal D),
        \]
        Note that both values could be infinite.
    \end{remark}
    
    
    \begin{lemma}\label{NPROPERTIES2}
        Let $A$ be a compact (pseudo-)invariant subset of $X$ such that $A \cap \widetilde X \neq \emptyset$, and let $\mathcal C$ be an $A$-open cover. Then, for every $j \geq 0$,
        \[
            \aleph_{A \cap \widetilde X} \big(f^{-j}\big(\mathcal C \setminus \Delta\big)\big) \leq \aleph_{A \cap \widetilde X}(\mathcal C) < \infty.
        \]
    \end{lemma}
    \begin{proof}
        Fix $j \geq 0$. As $A$ is compact, then there exists a finite sub-collection of $\mathcal C$ which is an $A$-open cover. Therefore, from Lemma \ref{FINITECOVERS}, we have that $f^{- j}(\mathcal C \setminus \Delta)$ admits a finite sub-cover of $A \cap \widetilde X$ (not necessarily $A$-open). Thus, part ({\it d}) of Lemma \ref{NPROPERTIES} and Remark \ref{COVERS.N} imply what is required, concluding the proof.
    \end{proof}
    
    \begin{remark}\label{LORC}
        Fix $n \geq 1$ and consider the same hypotheses as in Lemma \ref{NPROPERTIES2}. Then, we can always find a sub-collection of $\mathcal C^n$ that is a finite $A$-open cover of $A\cap\widetilde X$ thanks to Lemma \ref{OPENCOVER}. Thus, $\aleph_{A \cap \widetilde X}(\mathcal C^n) < \infty$ for every integer $n \geq 1$.
    \end{remark}

\section{Definition of topological entropy}\label{sec4}
    In this section, we define the concept of topological entropy adapted to pc-maps. The following result shows that such concept is well-defined.
    
    \begin{lemma}\label{COVERENTROPY}
        Let $A$ be a compact (pseudo-)invariant subset of $X$ such that $A\cap\widetilde X\neq\emptyset$ and let $\mathcal C$ be an $A$-open cover. Then, the sequence $\{a_n\}_{n\geq1}\subset\mathbb R$ given by
        \[
            a_n := \log\aleph_{A \cap \widetilde X}(\mathcal C^n) \qquad \forall \, n \geq 1,
        \]
        is sub-additive; that is $a_{n + k} \leq a_n + a_k$ for all $n, k \geq 1$.
    \end{lemma}
    \begin{proof}
        Fixing $n, k \geq 1$, from parts ({\it c})--({\it d}) of Lemma \ref{NPROPERTIES}, we have that
        \setlength \arraycolsep{2pt}
        \begin{eqnarray}
            a_{n + k} = \log \aleph_{A\cap\widetilde X}(\mathcal C^{n + k}) &=& \log \aleph_{A \cap \widetilde X}\!\!\left(\mathcal C^n \vee \bigvee_{j = n}^{n + k - 1} f^{-j}(\mathcal C \setminus \Delta)\right)
            \nonumber
            \\
            &\leq& \log \aleph_{A \cap \widetilde X}(\mathcal C^n) + \log\aleph_{A \cap \widetilde X}\!\!\left(\bigvee_{j = 0}^{k - 1} f^{-j}\big(\mathcal C \setminus \Delta\big)\!\right)\!. \label{INEQ1}
        \end{eqnarray}
        This implies that $a_{n + k}\leq a_n + a_k$, concluding the proof.
    \end{proof}
    
    \noindent Note that Lemma \ref{COVERENTROPY} guarantees that the limit
    \[
        \lim_{n \to \infty} \frac{\log\aleph_{A \cap \widetilde X}(\mathcal C^n)}{n}
    \]
    exists. Now, we proceed to define the concept of topological entropy for pc-maps with respect to an open cover.
    \begin{definition}
        Let $A$ be a compact (pseudo-)invariant subset of $X$ such that $A \cap \widetilde X \neq \emptyset$ and let $\mathcal C$ be a cover of $A \cap \widetilde X$ satisfying $\aleph_{A \cap \widetilde X}(\mathcal C) < \infty$. From Remark \ref{LORC.NOTOPEN}, we deduce that $\aleph_{\widetilde X}(\mathcal C^n) < \infty$ for every $n \geq 1$. Thus, we define the {\em topological entropy} of $f|_A$ with respect to $\mathcal C$ as
        \begin{equation}
            h_{top}(f|_A, \mathcal C) := \lim_{n \to \infty} \frac{\log \aleph_{A \cap \widetilde X}(\mathcal C^n)}{n} \geq 0. \label{HTOPC}
        \end{equation}
        In particular, note that $\aleph_{A \cap \widetilde X}(\mathcal C) < \infty$ when $\mathcal C$ is an $A$-open cover.
    \end{definition}
    
    \begin{lemma}\label{TOPOCOVERS}
        Let $A$ be a compact (pseudo-)invariant subset of $X$ such that $A \cap \widetilde X \neq \emptyset$. \\[-2ex]
        \begin{enumerate}
            \item[({\it a})] If $\mathcal C$, $\mathcal D$ are $A$-open covers such that $\mathcal C \preccurlyeq \mathcal D$, then $h_{top}(f|_{A}, \mathcal C) \leq h_{top}(f|_{A}, \mathcal D)$.
            \item[({\it b})] Let $B$ be a compact (pseudo-)invariant subset of $X$ such that $A \subset B$, and let $\mathcal C$ be a $B$-open cover. Then, $h_{top}(f|_A, \mathcal C \cap A) \leq h_{top}(f|_B, \mathcal C)$.
        \end{enumerate}
    \end{lemma}
    \begin{proof}
        \noindent({\it a}) This is a straightforward consequence of part ({\it a}) of Lemma \ref{NPROPERTIES}.\\[1ex]
        \noindent({\it b}) If $A \subset B$, then every $B$-open cover $\mathcal C$ defines an $A$-open cover $\mathcal C \cap A$. Therefore, it follows that $\aleph_{A \cap \widetilde X}\big((\mathcal C \cap A)^n\big)\leq \aleph_{B \cap \widetilde X}(\mathcal C^n)$ for all $n \geq 1$, which concludes the proof.
    \end{proof}
    
    Before proceeding to give a formal definition of entropy note that, by Remark \ref{REMARK.CAP}, if $\mathcal C$ is an $X$-open cover of $A$, then
    \[
        \aleph_{A \cap \widetilde X}(\mathcal C^n) = \aleph_{A \cap \widetilde X} \big((\mathcal C \cap A)^n\big) \qquad \forall \, n \geq 1.
    \]

    \vspace*{-2ex}
    \begin{definition}\label{ENTROPY.DEF}
        Let $A$ be a compact (pseudo-)invariant subset of $X$ such that $A \cap \widetilde X\neq\emptyset$. We define the {\em topological entropy} of $f|_A$ as
        \setlength \arraycolsep{2pt}
        \begin{eqnarray*}
            h_{top}(f|_A) := \sup\!\big\{h_{top}(f|_{A}, \mathcal C)\ :\ \mathcal C \text{ is an $A$-open cover}\big\} \geq 0.
        \end{eqnarray*}
    \end{definition}
    Note that, due to the compactness of $A$, one can take the supremum over all finite $A$-open covers in the previous definition. This follows from part ({\it a}) of Lemma \ref{TOPOCOVERS}.

    \begin{remark} \label{INDEP.HTOP}
        While $f$ is defined over the entire interval $X$, the definition of $h_{top}(f|_A)$ does not depend on how $f$ is defined over $A \cap \Delta$. That is, if $f$ and $g$ are pc-maps such that $\Delta_f = \Delta_g$ and $f = g$ on $X \setminus \Delta_f$, then $h_{top}(f|_A) = h_{top}(g|_A)$ for every compact (pseudo-)invariant subset $A\subset X$.
    \end{remark}

\section{Properties of the topological entropy} \label{sec6}
    In this section, we prove each part of Theorem \ref{ENTROPY}. First, note that its part (1) follows from the definition of entropy (it suffices to take $\Delta = \emptyset$).
    
    \begin{lemma}[{[Part (2) of Theorem \ref{ENTROPY}]}]\label{INVARIANTENTROPY}
        If $A\subset X$ is a compact and (pseudo-) invariant subset of $X$ such that $A\cap\widetilde X\neq\emptyset$, then $h_{top}(f|_{A})\leq h_{top}(f)$.
    \end{lemma}
    \begin{proof}
        Let $\mathcal C$ be an $X$-open cover. Then, by part ({\it b}) of Lemma \ref{TOPOCOVERS} we know that $h_{top}(f|_{A}, \mathcal C \cap A)\leq h_{top}(f, \mathcal C)$, which implies that $h_{top}(f|_{A})\leq h_{top}(f)$.
    \end{proof}
    
    \begin{lemma}[{[Part (3) of Theorem \ref{ENTROPY}]}]\label{CONJ}
        Let $f: X \to X$ and $g: Y \to Y$ be pc-maps defined on compact intervals $X$ and $Y$, respectively. If $f$ is topologically semi-conjugate (respectively, conjugate) to $g$, then $h_{top}(f) \geq h_{top}(g)$ (respectively, $h_{top}(f) = h_{top}(g)$).
    \end{lemma}
    
    \begin{proof}
        Let $n \geq 1$ be an integer and assume that $f$ is topologically semi-conjugate to $g$. Also, let $\varphi: X \to Y$ be a continuous surjection such that $\Delta_f =\varphi^{- 1}(\Delta_g)$ and $\varphi \circ f = g \circ \varphi$ on $X \setminus \Delta_f$. If $\mathcal B$ is a $Y$-open cover, we have that $\varphi^{-1}(\mathcal B)$ is an $X$-open cover. Moreover, for every collection $B_0, \ldots, B_{n - 1} \in \mathcal B$, we have that
        \[
            \bigcap_{j = 0}^{n - 1} f^{-j}\!\left(\varphi^{-1}(B_j) \setminus \Delta_f\right) \subset
            \varphi^{-1} \!\! \left(\bigcap_{j = 0}^{n - 1} g^{-j}(B_j \setminus \Delta_g)\!\right).
        \]
        This implies that $\varphi^{-1}(\mathcal B^n_g) \preccurlyeq \left(\varphi^{-1}(\mathcal B)\right)^n_f$, where both are $X$-open covers of $\widetilde X$ having a sub-collection that is a finite $X$-open cover of $\widetilde X$. Then, by part ({\it a}) of Lemma \ref{NPROPERTIES} we deduce that
        \begin{equation}
            \aleph_{\widetilde X} \! \left(\varphi^{-1}(\mathcal B^n_g)\right) \leq \aleph_{\widetilde X} \! \left(\left(\varphi^{-1}(\mathcal B)\right)^n_f\right)\!. \label{ALEPH1}
        \end{equation}
        Next, as $\varphi$ is a continuous surjection, we have that a $Y$-open sub-collection of $\mathcal B_g^n$ covers $\widetilde Y:= \widetilde Y_g$ if and only if its preimage with respect to $\varphi$ is an $X$-open cover of $\widetilde X$. Thus,
        \[
            \aleph_{\widetilde Y}(\mathcal B^n_g) = \aleph_{\widetilde X} \! \left(\varphi^{-1}(\mathcal B^n_g)\right).
        \]
        Next, since $n \geq 1$ is arbitrary, by \eqref{ALEPH1} we have that
        \[
            \aleph_{\widetilde Y}(\mathcal B^n_g) \leq \aleph_{\widetilde X} \! \left(\left(\varphi^{-1}(\mathcal B)\right)^n_f\right) \quad \forall \, n \geq 1,
        \]
        which implies $h_{top}(g, \mathcal B)\leq h_{top}\!\left(f, \varphi^{-1}(\mathcal B)\right)$. Finally, as $\mathcal B$ is an arbitrary $Y$-open cover, we conclude that
        \begin{eqnarray*}
            h_{top}(g) \leq \sup\!\left\{h_{top}\!\left(f, \varphi^{-1}(\mathcal B)\right)\ :\ \text{$\mathcal B$ is a $Y$-open cover}\right\} \leq h_{top}(f).
        \end{eqnarray*}
        Now, if in addition $f$ and $g$ are topologically conjugate, we also have that $g$ is topologically semi-conjugate to $f$. Therefore, the inequality $h_{top}(f) \leq h_{top}(g)$ also holds.
    \end{proof}

    \begin{lemma}[{[Part (4) of Theorem \ref{ENTROPY}]}]\label{PROP.ENTROPY.4}
         $h_{top}(f^k) = k \, h_{top}(f)$ for every integer $k \geq 0$.
    \end{lemma}
    \begin{proof}
        First, if $k = 0$, the result follows as $h_{top}(\id) = 0$. Now, fix $k \geq 1$. Let $\mathcal C$ be an $X$-open cover. Therefore, for all $n \geq 1$,
        \begin{align*}
            \mathcal C_{f^k}^{n} &\preccurlyeq \mathcal C_{f^k}^n \vee \left[\bigvee_{\begin{subarray}{c} \scriptscriptstyle 1 \leq j \leq n k - 1  \\ \scriptscriptstyle k \, \nmid \, j \end{subarray}} \!f^{-j}(\mathcal C \setminus \Delta_f)\right] = \mathcal C_f^{n k}
        \end{align*}
        where $k \nmid j$ means that $k$ does not divide $j$. Now, from part ({\it a}) of Lemma \ref{NPROPERTIES} we deduce that
        \setlength \arraycolsep{2pt}
        \begin{eqnarray*}
            \frac{1}{k} \lim_{n \to \infty} \frac{\log \aleph_{\widetilde X} \! \big(\mathcal C_{f^k}^{n}\big)}{n} \leq \lim_{n \to \infty} \frac{\log \aleph_{\widetilde X}\!\! \, \big(\mathcal C_f^{n k}\big)}{n k} &=& h_{top}(f, \mathcal C) \leq h_{top}(f).
        \end{eqnarray*}
        Thus, we conclude that
        \[
            \frac{h_{top}(f^k, \mathcal C)}{k} \leq h_{top}(f).
        \]
        Thus, as $\mathcal C$ is an arbitrary $X$-open cover, then $h_{top}(f^k) \leq k \, h_{top}(f)$.\\[-2ex]

        Reciprocally, if $k \geq 1$ and $\mathcal C$ is an $X$-open cover, from Lemma \ref{OPENCOVER} and Remark \ref{LORC.NOTOPEN} we deduce that
        \[
            \mathcal B:= \bigvee_{i = 0}^{k - 1} f^{- i}(\mathcal C \setminus \Delta_f) = \left(\bigvee_{i = 0}^{k - 1} f^{- i}(\mathcal C)\right)\setminus\Delta^k_f
        \]
        is an $X$-open cover of $\widetilde X$ satisfying $\aleph_{\widetilde X}(\mathcal B) < \infty$. Note that the only points in $X$ that $\mathcal B$ does not cover are those given by $\Delta_f^k$. To construct an $X$-open cover, it suffices to add $\mathcal B$ a finite number of open intervals:
        \[
            \mathcal B(\epsilon):= \mathcal B \cup \big\{(y - \epsilon, y + \epsilon) \cap X \ : \ y \in \Delta^k_f\big\} \qquad \forall \, \epsilon > 0.
        \]        
        From the compactness of $X$, $\aleph_{\widetilde X}\big(\mathcal B(\epsilon)\big) < \infty$ for every $\epsilon\geq0$ (where $\mathcal B(0) := \mathcal B$). Also, note that $\mathcal B(\epsilon) \preccurlyeq \mathcal B(\eta)$ whenever $\epsilon \geq \eta$. Then, from Lemma \ref{Xtilde.F} and part (a) of Lemma \ref{TOPOCOVERS}, we have
        \begin{align}
            h_{top}(f^k) \geq \sup_{\epsilon>0} \big\{h_{top}\big(f^k, \mathcal B(\epsilon)\big)\big\} &= h_{top}(f^k, \mathcal B) = \lim_{n \to \infty} \frac{\log \aleph_{\widetilde X}\big(\mathcal B^n_{f^k}\big)}{n} \nonumber
            \\
            &\qquad = \lim_{n \to \infty} \frac{1}{n} \log \aleph_{\widetilde X} \!\! \left(\bigvee \limits_{j = 0}^{n - 1}f^{- j k}\big(\mathcal B \setminus \Delta_{f^k}\big)\right) \nonumber
            \\
            &\qquad = \lim_{n \to \infty} \frac{1}{n} \log \aleph_{\widetilde X} \!\! \left(\left[\bigvee \limits_{j = 0}^{n - 1} f^{- j k}(\mathcal B)\right] \setminus \big(\Delta_{f^k}\big)^{\!n}_{\!}\right) \!\!. \label{VIC.1}
        \end{align}
        Now, it is not difficult to see that:
        \begin{equation}
            \bigvee_{j = 0}^{n - 1} f^{- j k}(\mathcal B) = \left(\bigvee_{\ell = 0}^{n k - 1} f^{- \ell}(\mathcal C)\right)\setminus\Delta_f^{n k} \qquad \text{and} \qquad \big(\Delta_{f^k}\big)^{\!n}_{\!} = \Delta_f^{n k}. \label{VIC.2}
        \end{equation}
        Thus, from \eqref{VIC.1} and \eqref{VIC.2}, we deduce that
        \[
            h_{top}(f^k) \geq k \cdot \lim_{n \to \infty} \frac{1}{n k} \log \aleph_{\widetilde X} \!\! \left(\left[\bigvee_{\ell = 0}^{n k - 1}f^{- \ell}(\mathcal C)\right] \setminus \Delta_f^{n k}\right) = k \cdot h_{top}(f),
        \]
        concluding the proof.
    \end{proof}

    To prove the following main result, it will be necessary to make use of the following elementary lemma.
    
    \begin{lemma}\label{ELEMENTARY}
        Let $p \geq 1$. Suppose that for every $i\in \{1, \ldots, p\}$, $\{a_{i, n}\}_{n \geq 1}$ is a sequence of real numbers greater than or equal to 1 such that 
        \[
            \lim_{n \to \infty} \frac{\log a_{i, n}}{n}
        \]
        exists and is equal to $a_i \in \mathbb R$. Then,
        \[
            \lim_{n \to \infty} \frac{1}{n} \log \! \left(\sum_{i = 1}^p a_{i, n}\right) = \max_{1 \leq i \leq p}\{a_i\}.
        \]
    \end{lemma}
    \begin{proof}
        The proof is based on induction and a basic lemma that is found just before Theorem 4 in \cite{AKM65}.
    \end{proof}
    
    \begin{lemma}[{[Part (5) of Theorem \ref{ENTROPY}]}]\label{QUEQUE}
        Let $p \geq 1$. Suppose that $X = A_1 \cup \ldots \cup A_p$, where $A_i$ is a compact (pseudo-)invariant subset of $X$ such that $A_i \cap \widetilde X \neq \emptyset$ for every $1 \leq i \leq p$. Then,
        \[
            h_{top}(f) = \max_{1 \leq i \leq p}\!\big\{h_{top}(f|_{A_i})\big\}.
        \]
    \end{lemma}
    \begin{proof}
        From Lemma \ref{INVARIANTENTROPY} we deduce that
        \[
            \max_{1 \leq i \leq p}\!\big\{h_{top}(f|_{A_i})\big\}\leq h_{top}(f).
        \]
        Now, if $\mathcal C$ is an $X$-open cover, we have that
        \[
            \log\aleph_{\widetilde X}(\mathcal C^n) \leq \log\!\left(\sum_{i = 1}^{p}\aleph_{A_i \cap \widetilde X}\big((\mathcal C \cap A_i)^n\big)\!\right) \qquad \forall \, n \geq 1.
        \]
        Therefore, from Lemma \ref{ELEMENTARY} we deduce that
        \[
            h_{top}(f, \mathcal C)\leq \max_{1\leq i\leq p} \!\big\{ h_{top}(f|_{A_i},\mathcal C\cap A_i)\big\}.
        \]
        Thus, it is enough to take supremum over all the $X$-open covers to obtain the desired result.
    \end{proof}

    To finish the proof of Theorem \ref{ENTROPY}, we will use the following lemma from topology:
    
   \vspace*{-2ex} 
    \begin{lemma}[{[Lebesgue's number, \cite{M00}]}]\label{LEBESGUE}
        Let $A\subset X$ be a compact set, and let $\mathcal C$ be an arbitrary $X$-open cover of $A$. Then, there exists $\delta > 0$ (Lebesgue's number) such that for every $x\in A$, there exists $C \in \mathcal C$ such that $B(x; \delta) \subset C$.
    \end{lemma}

    With this, we are ready to prove the last part of Theorem \ref{ENTROPY}.
    
    \begin{lemma}[{[Part (6) of Theorem \ref{ENTROPY}]}]\label{CKCK}
        Suppose that $\{\mathcal C_k\}_{k \geq 1}$ is a sequence of $X$-open covers such that $\diam(\mathcal C_k) \to 0$ as $k \to \infty$. Therefore,
        \begin{enumerate}
            \item[({\it a})] If $h_{top}(f) < \infty$, then $\lim_{k \to \infty} h_{top}(f, \mathcal C_k)$ exists and is equal to $h_{top}(f)$.
            \item[({\it b})] If $h_{top}(f) = \infty$, then $\lim_{k \to \infty} h_{top}(f, \mathcal C_k) = \infty$.
        \end{enumerate}
    \end{lemma}
    \begin{proof}
        ({\it a}) Assume that $h_{top}(f) < \infty$. Given $\epsilon > 0$, choose an $X$-open cover $\mathcal C$ such that
        \begin{equation}
            h_{top}(f, \mathcal C) > h_{top}(f) - \epsilon. \label{SUP.INEQ}
        \end{equation}
        Let $\delta$ be a Lebesgue number for $\mathcal C$. Now, choose $k^*\geq1$ such that $\diam(\mathcal C_k) < \delta$ for every $k\geq k^*$. Thus, for any $D\in\mathcal C_k$, with $k \geq k^*$, there exists $C\in\mathcal C$ such that $D\subset C$, which implies that $\mathcal C \preccurlyeq \mathcal C_k$ for every $k\geq k^*$. From \eqref{SUP.INEQ} and part ({\it a}) of Lemma \ref{TOPOCOVERS} we deduce that
        \[
            h_{top}(f) \geq h_{top}(f, \mathcal C_k) > h_{top}(f) -\epsilon.
        \]
        As $\epsilon$ was arbitrary, this concludes the proof of part ({\it a}).
        \\[1ex]
        \noindent ({\it b}) If $h_{top}(f) = \infty$, let $\rho > 0$ and choose an $X$-open cover $\mathcal C$ such that $h_{top}(f, \mathcal C) > \rho$. Following the same approach as in the proof of part (a), it can be shown that $\lim_{k \to \infty} h_{top}(f, \mathcal C_k) = \infty$. 
    \end{proof}


\section{Using Bowen's formula to compute the topological entropy}\label{sec5}
    Let us recall that P. Walters provided an example in which the value of the topological entropy calculated using Bowen's formula depends on the metric used. In particular, the value of the entropy can be different for metrics that are even equivalent (see \cite[page 171]{Wa00}). Walters' example results from the non-compactness of the space where the system is defined.

    This section is devoted to proving Theorem \ref{BOWEN.MAIN}, which establishes equality \eqref{LIMITS.BOWEN.MAIN} and that our concept of topological entropy does not depend on the choice of the metric (as long as it generates the induced topology on the respective invariant space where the system is defined). In short, we relate our concept of topological entropy with Bowen's formula for the map restricted to $\widetilde X$.  \\[-2ex]
    
    We begin this section by stating the so-called {\em dynamical distance} or {\em Bowen metric} on $X$. For all $n \geq 1$, we consider the metric on $X$ defined by the map $\rho_n: X \times X \to [0, \infty)$ such that
    \[
        \rho_n(x, y) := \max_{0 \leq j \leq n - 1}\big\{\dist\big(f^j(x), f^j(y)\big)\big\} \qquad \forall \, (x, y) \in X \times X,
    \]
    where ``$\dist$'' is the Euclidean metric. We highlight that it is possible to consider any other metrics equivalent to $\dist$, as it will be proved later (see the comment following the proof of Theorem \ref{BOWEN.FORMULA}).
    
    We denote the open ball in $(X, \rho_n)$ (respectively, $(X, \dist)$) of center $z \in X$ and radius $\eta > 0$ by $B^{\rho_n}(z, \eta)$ (respectively, $B(z, \eta)$). The set $B^{\rho_n}(z, \eta)$ is commonly referred to as a {\em dynamical ball}. As usual, note that $B^{\rho_n}(z, \eta)$ depends on how $f$ is defined over $\Delta$. To relate the concepts in this Section with our definition of entropy, we consider a modified ball as follows.

    \begin{definition}
        Consider $n \geq 1$, $\epsilon > 0$, and let $A$ be a (pseudo-)invariant subset of $X$ such that $A \cap \widetilde X \neq \emptyset$. We define the {\em modified dynamical ball} of center $z \in X$ and radius $\eta > 0$ as
        \begin{equation}
            Q_A^n(z; \eta):= B^{\rho_n}(z; \eta) \cap \bigcap_{j = 0}^{n - 1} f^{- j}(A \setminus \Delta). \label{INTEROPEN1}
        \end{equation}
    \end{definition}
    
    
    Observe that the compactness of $A$ is not required for this definition nor most of the results below. 
    Next, we show that the set defined by \eqref{INTEROPEN1} is an open set of $A$.
    
    
    \begin{lemma}
        Consider $n \geq 1$, $\epsilon > 0$, and a (pseudo-)invariant subset $A \subset X$ such that $A \cap \widetilde X \neq \emptyset$. Then, for every $x\in X$, $Q_A^n(x; \epsilon)$ is a relative open set of $A$.
    \end{lemma}
    \begin{proof}
        Fix $n \geq 1$, $\epsilon > 0$ and $x \in X$. Then, $y \in Q_A^n(x; \epsilon)$ if and only if
        \[
            f^j(y) \in A \setminus \Delta \quad \land \quad \dist\big(f^j(x), f^j(y)\big) < \epsilon \qquad \forall \, j \in \{0, \ldots, n - 1\},
        \]
        which happens if and only if $f^j(y)\in \big(B\big(f^j(x); \epsilon\big) \cap A\big) \setminus \Delta$ for all $j \in \{0, \ldots, n - 1\}$. Thus, we deduce
        \begin{equation}
            Q_A^n(x; \epsilon) = \bigcap_{j = 0}^{n - 1} f^{-j}\big(\big(B(f^j(x); \epsilon) \cap A\big) \setminus \Delta\big). \label{NICOLE}
        \end{equation}
        Now, observe that,
        \begin{align*}
            \bigcap_{j = 0}^{n - 1} f^{- j}(A \setminus \Delta) = \bigcap_{j = 0}^{n - 1} \big(f^{- j}(A) \setminus f^{- j}(\Delta)\big) \subset A \setminus \Delta^n.
        \end{align*}
        Moreover, by using this together with \eqref{TECH.DETAIL}, we deduce
        \begin{equation}
            \bigcap_{j = 0}^{n - 1} f^{-j}(A \setminus \Delta) = A \setminus \Delta^n. \label{INTER.D}
        \end{equation}
        Therefore, from \eqref{NICOLE} and \eqref{INTER.D} we have
        \begin{equation}
            Q_A^n(x; \epsilon) = \left[A\cap\bigcap_{j = 0}^{n - 1} f^{-j} \big(B(f^{j}(x); \epsilon)\big)\right] \setminus \Delta^n. \label{ABIERTO}
        \end{equation}
        Last, as $\Delta^n$ is a finite set, we conclude that $Q_A^n(x;\epsilon)$ is a relative open set of $A$.
    \end{proof}
    
    Now, notice that
    \begin{align}
        \dist(x, y)\leq \rho_n(x, y)\leq \sum_{j = 0}^{n - 1} \dist\big(f^j(x), f^j(y)\big) \qquad \forall \, x, y \in X, \label{DANA1}
    \end{align}
    which implies that any Cauchy sequence with respect to $\rho_n$ is also a Cauchy sequence with respect to $\dist$. Furthermore, if $\{f^j(x_i)\}_{i \geq 0}$ is a Cauchy sequence with respect to $\dist$ for all $j \in \{0, \ldots, n - 1\}$, then the sequence $\{x_i\}_{i \geq 0}$ is a Cauchy sequence with respect to $\rho_n$. Now, we prove that the subspace $\big(A \cap \widetilde X, \rho_n\big)$ is totally bounded, which is equivalent to the following result.
    
    \begin{lemma}\label{SCOMPACT}
        Let $A \subset X$ be a (pseudo-)invariant subset such that $A \cap \widetilde X \neq \emptyset$. Then, for all $n \geq 1$, every sequence in $\big(A \cap \widetilde X, \rho_n\big)$ has a Cauchy subsequence.
    \end{lemma}
    
    \begin{proof}
        Consider a sequence $\{x_k\}_k \subset A \cap \widetilde X$, where the subscripts $k$ are non-negative integers. Clearly, $\{x_k\}_k \subset \overline A \subset X$, which implies that there exists a $\dist$-convergent subsequence $\{x_k^0\}_k \subset \{x_k\}_k$ within $\overline A$. Since $x_k^0 \in A \cap \widetilde X$ for every $k \geq 0$, we conclude that $\{x_k^0\}_k$ forms a $\dist$-Cauchy sequence. Similarly, as $\big\{f(x^0_k)\big\}_k \subset A\cap \widetilde X$, there exists a $\dist$-Cauchy subsequence $\{x_k^1\}_k \subset \{x^0_k\}_k$ such that $\big\{f(x^1_k)\big\}_k$ is a $\dist$-Cauchy subsequence of $\big\{f(x^0_k)\big\}_k$. By repeating this process iteratively, we obtain a $\dist$-subsequence $\{x_k^j\}_k$ of $\{x_k\}_k$ such that $\big\{f^j(x_k^j)\big\}_k\subset A\cap\widetilde X$ forms a Cauchy sequence with respect to $\dist$ for all $j \in \{0, \ldots, n-1\}$. Therefore, by \eqref{DANA1}, we have that $\{x_k^j\}_k \subset A \cap \widetilde X$ is a Cauchy sub-sequence of $\{x_k\}_k$ with respect to $\rho_n$.
    \end{proof}

    Next, we recall the classical concepts used to define Bowen's formula. Let $A \subset X$ be a (pseudo-)invariant subset such that $A \cap \widetilde X \neq \emptyset$. The set $E \subset A \cap \widetilde X$ is said to be an $(n, \epsilon)$-{\em separated set} of $A \cap \widetilde X$ if for all $x, y \in E$ with $x \neq y$, there exists $j \in \{0, \ldots, n - 1\}$ such that $\dist\big(f^{j}(x), f^j(y)\big) \geq \epsilon$. We denote the maximal cardinality of $(n, \epsilon)$-separated sets of $A \cap \widetilde X$ by $s_n(A; f, \epsilon)$, where
    \begin{equation}
        s_n(A; f, \epsilon) = \max\!\big\{\#E\ :\ \text{$E$ is an $(n, \epsilon)$-separated set of $A \cap \widetilde X$}\big\}.\label{PIERRE2}
    \end{equation}
    Now, we prove that this quantity is well-defined. 
    
    \begin{lemma}\label{SN.FINITE}
        Let $A$ be a (pseudo-)invariant subset of $X$ such that $A \cap \widetilde X \neq \emptyset$. Then, for every $n \geq 1$ and $\epsilon > 0$, $s_n(A; f, \epsilon) < \infty$.
    \end{lemma}
    \begin{proof}
        Let $n \geq 1$ and $\epsilon > 0$. We begin the proof assuming, by contradiction, that $s_n(f, \epsilon) = \infty$. This implies the existence of an $(n,\epsilon)$-separated set $E \subset A \cap \widetilde X$ such that $\#E = \infty$. Thus, there exists an infinite sequence $\{x_k\}_{k\geq 0}\subset E$ where each element is different from one another (i.e., $x_i \neq x_j$ for $i \neq j$). Since $\{x_k\}_{k \geq 0}\subset A \cap \widetilde X$, we can use Lemma \ref{SCOMPACT} to conclude that there exists a $\rho_n$-Cauchy subsequence $\{x_{k_j}\}_{j \geq 0}\subset E$. Consequently, there exists an integer $m\geq 1$ such that
        \[
            \rho_n(x_{k_i}, x_{k_j}) < \epsilon \qquad \forall \, i, j \geq m.
        \]
        However, this contradicts the fact that $E$ is an $(n, \epsilon)$-separated set of $A \cap \widetilde X$. Therefore, we conclude that $s_n(A; f, \epsilon) < \infty$.
    \end{proof}
    
    In a dual way, let $A \subset X$ be a (pseudo-)invariant subset such that $A \cap \widetilde X \neq \emptyset$, and let $E \subset A \cap \widetilde X$. The set $E$ is called $(n, \epsilon)$-{\em spanning set of $A\cap\widetilde X$} if for all $x \in A \cap \widetilde X$, there exists $y \in E$ such that $\dist\big(f^{j}(x), f^j(y)\big) < \epsilon$ for all $j \in \{0, \ldots, n - 1\}$, which is equivalent to
    \[
        A \cap \widetilde X \subset \bigcup_{y \in E} Q_A^n(y; \epsilon).
    \]
    Now, we prove that the collection of finite spanning sets is non-empty. 
    
    \begin{lemma}\label{RELCOM}
        Let $A\subset X$ be a (pseudo-)invariant subset such that $A \cap \widetilde X \neq \emptyset$. Then, for every $n \geq 1$ and $\epsilon > 0$, there exists a finite set $E \subset A \cap \widetilde X$ that is an $(n, \epsilon)$-spanning set of $A \cap \widetilde X$.
    \end{lemma}
    
    \begin{proof}
        Let $n \geq 1$ and $\epsilon > 0$. Since $\Delta^n$ is finite, then there exist open intervals $I_1, \ldots, I_m\subset X$ such that $I_j \cap A \cap \widetilde X \neq \emptyset$ for all $j \in \{1, \ldots, m\}$, and
        \[
            A \cap\widetilde X = (A \cap \widetilde X) \setminus \Delta^n = \bigcup_{j = 1}^m I_j\cap A\cap\widetilde X.
        \]
        It is clear that, if $x, y \in I_i\cap A\cap\widetilde X$ for some $i \in \{1, \ldots, m\}$, then $f^k(x)$ and $f^k(y)$ belong to the same continuity piece for each $k\in\{0, \ldots, n - 1\}$. Now, note that for every $j \in \{1, \ldots, m\}$, the collection $\big\{B(x, \epsilon)\ :\ x\in I_j\cap A\cap \widetilde X\big\}$ is an open cover of $\overline{I_j \cap A \cap \widetilde X}\subset X$. Therefore, there exist $x_0^j, x_1^j, \ldots, x_{p_j}^j \in I_j\cap A\cap \widetilde X$ such that
        \[
            I_j \cap A \cap \widetilde X \subset \overline{I_j\cap A \cap \widetilde X} \subset B(x_0^j, \epsilon) \cup \ldots \cup B(x_{p_j}^j, \epsilon).
        \]
        Thus, we have that the finite set $E := \bigcup_{j = 1}^m \{x_0^j, x_1^j, \ldots, x_{p_j}^j\}$ is an $(n, \epsilon)$-spanning set of $A\cap\widetilde X$, concluding the proof.
    \end{proof}
    
    Observe that, under the same assumptions as in Lemma \ref{RELCOM}, it makes sense to define the minimal cardinality of $(n, \epsilon)$-spanning sets of $A \cap \widetilde X$, by
    \[
        r_n(A; f, \epsilon) := \min\{\#E\ :\ \text{$E$ is an $(n, \epsilon)$-spanning set of $A\cap\widetilde X$}\} < \infty.
    \]
    
    Now, we prove some properties of comparison and monotonicity of $r_n(A; f, \epsilon)$ and $s_n(A; f, \epsilon)$.
    
    \begin{lemma}\label{LEMMAI}
        Consider $n \geq 1$, $\epsilon > 0$ and let $A\subset X$ be a (pseudo-)invariant subset such that $A \cap \widetilde X \neq \emptyset$. Then, the following statements hold:
        \\[-2ex]
        \begin{enumerate}
            \item[{\it (i)}] If $0 < \epsilon' < \epsilon$, then $r_n(A; f, \epsilon)\leq r_n(A; f, \epsilon')$ and $s_n(A; f, \epsilon)\leq s_n(A; f, \epsilon')$.
            \item[{\it (ii)}] $r_n(A; f, \epsilon)\leq s_n(A; f, \epsilon)\leq r_n(A; f, \epsilon/2)$.
        \end{enumerate}
    \end{lemma}
    \begin{proof}
        \noindent{\it (i)} Assume that $0 < \epsilon' < \epsilon$. Let $E' \subset A \cap \widetilde X$ be an $(n, \epsilon')$-spanning set of $A\cap\widetilde X$ such that $\#E' = r_n(A; f, \epsilon')$. We have that
        \[
            A \cap \widetilde X \subset \bigcup_{x \in E'} Q_A^n(x; \epsilon') \subset \bigcup_{x \in E'} Q_A^n(x; \epsilon).
        \]
        Thus, we deduce that $E'$ is an $(n, \epsilon)$-spanning set and, therefore, there exists an $(n, \epsilon)$-spanning set $E \subset E'$ such that $\#E = r_n(A; f, \epsilon)$. We conclude that $r_n(A; f, \epsilon) \leq r_n(A; f, \epsilon')$.
        
        Now, suppose that $E \subset A \cap \widetilde X$ is an $(n, \epsilon)$-separated set of $A \cap \widetilde X$ such that $\#E = s_n(A; f, \epsilon)$. If $x, y\in E$ are different points, then there exists $j \in \{0, \ldots, n - 1\}$ such that
        \[
            \dist\big(f^j(x), f^j(y)\big)\geq \epsilon >\epsilon'.
        \]
        Thus, we deduce that $E$ is an $(n, \epsilon')$-separated set of $A \cap \widetilde X$ and, therefore, there exists $E' \subset A \cap \widetilde X$, an $(n, \epsilon')$-separated set of $A \cap \widetilde X$ such that $E \subset E'$ and $\#E' = s_n(A; f, \epsilon')$. We conclude that $s_n(A; f, \epsilon) \leq s_n(A; f, \epsilon')$.
        \\[1ex]
        \noindent{\it (ii)} Let $E \subset A \cap \widetilde X$ be an $(n, \epsilon)$-separated set of $A \cap \widetilde X$ such that $\# E = s_n(A; f, \epsilon)$. Then, for all $y \in (A \cap \widetilde X) \setminus E$ we have that $E \cup \{y\}$ is not an $(n, \epsilon)$-separated set of $A \cap \widetilde X$. Therefore, there exists $x_y\in E$ such that
        \[
            \dist\big(f^j(x_y), f^j(y)\big) < \epsilon \qquad \forall \, j \in \{0, \ldots, n - 1\}.
        \]
        We deduce that $y \in Q_A^n(x_y; \epsilon)$. Since $y\in (A \cap \widetilde X) \setminus E$ is arbitrary, we deduce that
        \[
            A \cap \widetilde X \subset \bigcup_{x \in E} Q_A^n(x; \epsilon),
        \]
        which implies that $E$ is an $(n, \epsilon)$-spanning set of $A \cap \widetilde X$. Thus, we have that
        \[    
            r_n(A; f, \epsilon) \leq \#E = s_n(A; f, \epsilon).
        \]
        To prove the second inequality, consider $E \subset A \cap \widetilde X$, an $(n, \epsilon)$-separated set of $A \cap \widetilde X$, and $F \subset A \cap \widetilde X$, an $(n, \epsilon/2)$-spanning set of $A \cap \widetilde X$. Then, we have
        \[
            E \subset A \cap \widetilde X \subset \bigcup_{y \in F} Q_A^n(y; \epsilon/2).
        \]
        Now, define the map $\phi: E \to F$ as
        \[
            \phi(x) := \min\!\big\{y\in F\ :\ x \in Q_A^n(y; \epsilon/2)\big\} \qquad \forall \, x \in E.
        \]
        We claim that $\phi$ is an injective map. Indeed, suppose that $x, z\in E$ satisfy $\phi(x) = y = \phi(z)$. Then,
        \[
            \dist\big(f^j(x), f^j(z)\big) \leq \dist\big(f^j(x), f^j(y)\big) + \dist\big(f^j(y), f^j(z)\big) < \frac{\epsilon}{2} + \frac{\epsilon}{2} = \epsilon,
        \]
        for all $j \in \{0, \ldots, n - 1\}$. As $E$ is an $(n, \epsilon)$-separated set of $A \cap \widetilde X$, we deduce that $x = z$, which proves the claim. Thus, we have that $\#E\leq \#F$ and, as $E$ and $F$ are arbitrary, we conclude that $s_n(A; f, \epsilon) \leq r_n(A; f, \epsilon/2)$.
    \end{proof}
    
    Here, note that under the same assumptions as in Lemma \ref{LEMMAI}, its part {\it (i)} lets us conclude that
    \[
        \epsilon \mapsto \limsup_{n \to \infty} \frac{\log s_n(A; f, \epsilon)}{n} \quad \text{ and } \quad \epsilon \mapsto \limsup_{n \to \infty} \frac{\log r_n(A; f, \epsilon)}{n}
    \]
    are non-increasing maps (as long as the involved limits are finite). Thus, we conclude that
    \begin{equation}
        s(A; f) := \lim_{\epsilon \to 0} \limsup_{n \to \infty} \frac{\log s_n(A; f, \epsilon)}{n} \quad \text{ and } \quad r(A; f) := \lim_{\epsilon \to 0} \limsup_{n \to \infty} \frac{\log r_n(A; f, \epsilon)}{n}\label{PIERRE}
    \end{equation}
    always exist, considering that they could be infinite. Next, we prove that the topological entropy can be obtained via Bowen's formula under the assumption of compactness. The following result is a reformulation of Bowen's formula established in Theorem \ref{BOWEN.MAIN}.

    \begin{theorem} \label{BOWEN.FORMULA}
        Let $A\subset X$ be a compact (pseudo-)invariant subset such that $A \cap \widetilde X \neq \emptyset$, then
        $h_{top}(f|_A) = s(A; f) = r(A; f)$.
    \end{theorem}
    \begin{proof}
        Using part {\it (ii)} of Lemma \ref{LEMMAI} and a squeeze argument, we conclude that both limits in \eqref{PIERRE} are equal. Thus, it is sufficient to show that $s(A; f) \leq h_{top}(f|_A) \leq r(A;f)$ to conclude the desired result.
        
        Consider $n \geq 1$, $\epsilon > 0$, and let $E \subset A \cap \widetilde X$ be an $(n, \epsilon)$-separated set of $A \cap \widetilde X$ such that $\#E = s_n(A; f, \epsilon)$. Consider $\mathcal C$, an $A$-open cover such that $\diam(\mathcal C) < \epsilon$. Since $A$ is compact, we can assume that $\mathcal C$ is finite. Note that if $x, y\in A \cap \widetilde X$ belong to the same element of $\mathcal C^n$, then
        \[
            \dist \big(f^j(x), f^j(y)\big) < \max_{D \in \, \mathcal C}\{\diam(D)\} < \epsilon \qquad \forall \, j \in \{0, \ldots, n - 1\}.
        \]
        In particular, $\#(E \cap C) \leq 1$ for all $C\in\mathcal C^n$. Thus, we conclude that
        \[
            s_n(A; f, \epsilon) = \#E \leq \aleph_{A \cap \widetilde X}(\mathcal C^n) \qquad \forall \, n \geq 1.
        \]
        Therefore,
        \[
            \limsup_{n \to \infty} \frac{\log s_n(A; f, \epsilon)}{n} \leq \lim_{n\to\infty} \frac{\log \aleph_{A \cap \widetilde X}(\mathcal C^n)}{n} = h_{top}(f|_A, \mathcal C\cap A)\leq h_{top}(f|_A),
        \]
        which lets us conclude that $s(A; f) \leq h_{top}(f|_A)$, when taking the limit as $\epsilon \to 0$. \\[-2ex]
        
        To prove the second inequality, let $\mathcal C$ be an $X$-open cover of $A$ and let $\epsilon > 0$ be the Lebesgue's number associated with $\mathcal C$. Again, since $A$ is compact, we can assume that $\mathcal C$ is finite. Also, let $E \subset A \cap \widetilde X$ be an $(n, \epsilon)$-spanning set of $A \cap \widetilde X$ such that $\#E = r_n(A; f, \epsilon)$. Then, for all $x \in A \cap \widetilde X$ and $j \in \{0, \ldots, n - 1\}$, there exists $C_{x, j} \in \mathcal C$ such that
        \[
            \big[B\big(f^{j}(x), \epsilon\big) \cap A\big] \setminus \Delta \subset \big[C_{x, j} \cap A\big] \setminus \Delta.
        \]
        Then, by using \eqref{NICOLE}, we conclude that
        \[
            Q_A^n(x; \epsilon) = \bigcap_{j = 0}^{n - 1} f^{-j} \big([B(f^j(x); \epsilon) \cap A] \setminus \Delta\big) \subset \bigcap_{j = 0}^{n - 1} f^{-j}([C_{x, j} \cap A] \setminus \Delta).
        \]
        Since $E$ is an $(n, \epsilon)$-spanning set of $A\cap\widetilde X$, we deduce that the collection
        \[
            \mathcal T := \left\{\bigcap_{j = 0}^{n - 1} f^{-j}([C_{x, j} \cap A] \setminus \Delta)\in(\mathcal C \cap A)^n\ :\ x \in E\right\}
        \]
        is a finite $A$-open cover of $A \cap \widetilde X$. Thus, from part {\it (b)} of Lemma \ref{NPROPERTIES}, as $\mathcal T \subset (\mathcal C \cap A)^n$ we have that
        \[
            \aleph_{A \cap \widetilde X}(\mathcal C^n) \leq \#E = r_n(A; f, \epsilon) \qquad \forall \, n \geq 1.
        \]
        Therefore,
        \[
            h_{top}(f|_A; \mathcal C \cap A) = \lim_{n \to \infty} \frac{\log \aleph_{A \cap \widetilde X}(\mathcal C^n)}{n} \leq \limsup_{n \to \infty} \frac{\log r_n(A; f, \epsilon)}{n}.
        \]
        Now, taking the limit as $\epsilon \to 0$, we have that $h_{top}(f|_A; \mathcal C \cap A) \leq r(A; f)$. Since $\mathcal C$ is an arbitrary $X$-open cover of $A$, we conclude $h_{top}(f|_A) \leq r(A; f)$. Thus, $h_{top}(f|_A) = s(A; f) = r(A; f)$.
    \end{proof}

    Since $h_{top}(f)$ is a purely topological concept, the equality in Theorem \ref{BOWEN.FORMULA} shows that there is no dependence on the metric used to compute our entropy. In fact, Note that all the results in this section are valid for any choice of metric, as long as it is equivalent to $\dist$.

\section{Proof of Theorems \ref{main2} and \ref{main3}}\label{sec7}
    This section is devoted to proving Theorems \ref{main2} and \ref{main3}, together with Corollary \ref{GIORNO-GIOVANNA} (see Corollary \ref{COSA}). Without loss of generality, we assume that $f$ is {\it piecewise strictly monotonic} throughout this section; that is, the restriction of $f$ to each of its continuity pieces is injective (note that, as $f$ has a finite number of critical points, it is always possible to adapt the continuity pieces of $f$ so that the map is piecewise strictly monotonic). Furthermore, to enhance the clarity of the arguments, we assume that the points of $\Delta$ fall into two categories: jump discontinuities or critical continuity points of $f$ (that is, we do not allow for removable discontinuities). The proof of Theorem \ref{main2} is divided into 2 parts: Subsections \ref{MURIEL} and \ref{LORENA}; whilst the proof of Theorem \ref{main3} is given in Subsection \ref{GATALINA}.

    We begin by defining a special cover of $\widetilde X$, which will be key for the proofs we provide in this section. We define $\mathcal D := \{X_1, \ldots, X_N\}$, the {\em natural $X$-open cover} of $\widetilde X$, which fulfills $\mathcal D \cap \Delta = \{\emptyset\}$ and, therefore,
    \[
        \mathcal D^n = \bigvee_{j = 0}^{n - 1} f^{-j}(\mathcal D) \qquad \forall \, n \geq 1.
    \]
    In addition, for an arbitrary collection $\mathcal C$ of $X$, we define the {\em topological boundary} of $\mathcal C$, denoted by $\partial \mathcal C$, as the topological boundary (with respect to $X$) of the union of its elements. In particular, observe that $\partial \mathcal D = \Delta$.
    
    \subsection{Preimages of critical points, monotonicity intervals, and entropy} \label{MURIEL}
        We state a straightforward result using that, for every $k\in \{0, 1, \ldots, n\}$, $f^k$ is a continuous and strictly monotonic function on the connected components of $X\setminus \Delta^n$. In particular, when $k = n$, these components coincide with the continuity pieces of $f^n$. This idea serves as a straightforward proof of the following result.
        \begin{lemma}\label{INTCAP}
            Let $n \geq 1$. Then, for every open interval $U\subset X$ and connected component $V$ of $X\setminus\Delta^n$, the set $f^{-k}(U)\cap V$ is either $\emptyset$ or an open interval for each $k \in \{0, 1, \ldots, n\}$.
        \end{lemma}  

        \begin{remark}\label{DNXD}
            $\mathcal D^n$ is a collection of open pairwise disjoint intervals for every $n\geq 1$. In what follows, we prove that $\partial(\mathcal D^n) = \Delta^n$, which will allow us to deduce that $\mathcal D^n$ coincides with the collection of all connected components of $X \setminus \Delta^n$.
        \end{remark}
    
        \begin{lemma} \label{laweapelua}
            For every $n\geq 0$, $f^{-n}(\Delta)\subset\partial \big(f^{-n}(\mathcal D)\big) \subset \Delta^{n + 1}$.
        \end{lemma}
        \begin{proof}
            Let $n \geq 0$ and $y \in f^{-n}(\Delta)$.
            This implies that $f^n(y) \in \Delta$. Also, let $U$ be an arbitrary neighbourhood of $y$. Then, it is clear that
            \begin{align*}
                U \cap \left(X \setminus \bigcup_{i = 1}^N f^{-n}(X_i)\right)
                = U \cap f^{-n}(\Delta) \neq \emptyset.
            \end{align*}
            Furthermore, from the density of $X \setminus f^{-n}(\Delta)$ in $X$, we deduce that
            \begin{align*}
                \emptyset \neq U \cap \big(X \setminus f^{-n}(\Delta)\big) = U \cap \left(\bigcup_{i = 1}^N f^{-n}(X_i)\right)\!.
            \end{align*}
            Therefore, we have that $y \in \partial \big(f^{-n}(\mathcal D)\big)$, which implies $f^{-n}(\Delta) \subset \partial \big(f^{-n}(\mathcal D)\big)$.\\[-2ex]
            
            For the second inclusion, let $y \notin \Delta^{n + 1}$. Therefore, $y$ belongs to a connected component of $X \setminus \Delta^{n + 1}$, which is an open interval. Let $\epsilon > 0$ be such that $V := (y - \epsilon, y + \epsilon)$ is completely contained in such connected component of $X \setminus \Delta^{n + 1}$. Therefore,
            \begin{align*}
                V \cap \left(X \setminus \bigcup_{i = 1}^N f^{-n}(X_i)\right) = V \cap f^{-n}(\Delta) = \emptyset,
            \end{align*}
            which proves that $y \notin \partial \big(f^{-n}(\mathcal D)\big)$.
        \end{proof}
    
        A consequence of this is the following result:
        \begin{lemma}\label{PARTIAL.FND}
            For every $n \geq 1$, $\partial(\mathcal D^n) = \Delta^n$.
        \end{lemma}
        \begin{proof}
            We proceed by induction on $n$. First, if $n = 1$, the result follows. Next, let us assume that $\partial\big(\mathcal D^k\big) = \Delta^k$ for $k \geq 1$. Thus, by Lemmas \ref{INTCAP} and \ref{laweapelua}, in addition to the induction hypothesis, we have that $\mathcal D^{k + 1} = \mathcal D^k \vee f^{-k}(\mathcal D)$. Therefore,
            \begin{align*}
                 \partial \big(\mathcal D^{k + 1}\big) = \partial \big(\mathcal D^k\big) \cup \partial \big(f^{-k}(\mathcal D)\big) = \Delta^{k + 1},
            \end{align*}
            which concludes the proof.
        \end{proof}
    
        \begin{lemma}\label{CARDI.N}
            For every $n \geq 1$, $\aleph_{\widetilde X}(\mathcal D^n) = \#(\mathcal D^n)$.
        \end{lemma}
    
        \begin{proof}
            This is a straightforward consequence of the fact that $\mathcal D^n$ is a pairwise disjoint collection of intervals, for every $n \geq 1$.
        \end{proof}
    
        We recall that, for every $n \geq 1$, $c_n \geq 1$ is the smallest number of intervals on which $f^n$ is monotonic and essentially continuous (see \eqref{Emily}). Since $f$ is piecewise strictly monotonic and the map does not allow for removable discontinuities, it follows that $c_n$ coincides with the cardinality of the collection of all connected components of $X \setminus \Delta^n$; that is, $c_n = \#(\mathcal D^n)$. With this, we prove the first part of Theorem \ref{main2}.

        \begin{lemma} \label{CEPSILON}
        The following inequality holds:
        \[
            h_{top}(f) \geq \lim_{n \to \infty} \frac{1}{n} \log c_n.
        \] 
        \end{lemma}
        \begin{proof}
            For each $\epsilon > 0$, denote by $\mathcal C_\epsilon$ the open cover of $X$ given by $\mathcal D \cup \mathcal D_\epsilon$, where
            \begin{align*}
                \mathcal D_\epsilon := \bigcup_{d\in\Delta}\big\{(d - \epsilon, d + \epsilon)\cap X\big\}.
            \end{align*}
            Naturally, it is convenient to denote $\mathcal C^n_0:= \mathcal D^n$ for every positive integer $n \geq 1$.
            
            Next, if $\mathbb N$ denotes the set of positive integers starting at $1$, let $a: [0, \infty) \times \mathbb N \to [0, \infty)$ be the function given by
            $$
                a(\epsilon, n):= \frac{\log \aleph_{\widetilde X}(\mathcal C_\epsilon^n)}{n} \qquad \forall \,(\epsilon, n) \in [0, \infty) \times \mathbb N.
            $$
            Now, note that for every $\xi > \eta \geq 0$, we have that $\mathcal C_\xi \preccurlyeq \mathcal C_\eta \preccurlyeq \mathcal C_0 = \mathcal D$, which implies that $a(\xi, n) \leq a(\eta, n) \leq a(0, n)$ for all $n \geq 1$. In particular, this lets us ensure that $\lim \limits_{n \to \infty} a(\xi, n) \leq \lim\limits_{n \to \infty} a(\eta, n)$. Next, we define $g: [0, \infty) \to [0, \infty)$ by
            \begin{align*}
                g(\epsilon):= \lim_{n \to \infty} a(\epsilon, n) \qquad \forall \, \epsilon \geq 0.
            \end{align*}
            This function is clearly well-defined. Also, $g$ is a decreasing map and, therefore, $\sup \limits_{\epsilon \geq 0} \{g(\epsilon)\} = g(0)$. Thus, we have that\\[-2ex]
            \begin{align}
                h_{top}(f) \geq \sup_{\epsilon \geq 0}\!\big\{h_{top}(f, \mathcal C_\epsilon)\big\} = g(0) = \lim_{n \to \infty} \frac{\log \aleph_{\widetilde X}(\mathcal D^n)}{n}. \label{EDGARDO}
            \end{align}
            Finally, from Lemma \ref{CARDI.N}, we have that $c_n = \#(\mathcal D^n) = \aleph_{\widetilde X}(\mathcal D^n)$ for every $n \geq 1$. Therefore,
            \begin{equation}
                \lim_{n \to \infty} \frac{\log \aleph_{\widetilde X}(\mathcal D^n)}{n} = \lim_{n \to \infty} \frac{\log c_n}{n}. \label{DANA}
            \end{equation}
            Thus, from \eqref{EDGARDO} and \eqref{DANA}, we conclude the result.
        \end{proof}

    \subsection{Entropy conditioned to a cover and conditional entropy} \label{LORENA}
        The ideas and definitions in this sub-section are adapted versions of what is defined in \cite{M76,MS80} to our context.
        
        Let $\mathcal B, \mathcal C$ be two covers of $\widetilde X$ such that $\#\mathcal B < \infty$ and $\aleph_{\widetilde X}(\mathcal C) < \infty$. Then, we can define
        \[
            \aleph_{\widetilde X}(\mathcal C | \mathcal B) := \max_{B \in \mathcal B} \aleph_{B \cap \widetilde X}(\mathcal C),
        \]
        where, by definition, we have that $\aleph_\emptyset(\mathcal C) = 1$. In particular, note that $\aleph_{\widetilde X}(\mathcal C | \{X\}) = \aleph_{\widetilde X}(\mathcal C)$. Also, if $\mathcal C$ is an $X$-open cover, then $\aleph_{\widetilde X}(\mathcal C | \mathcal B)$ is well-defined for every finite cover $\mathcal B$ of $\widetilde X$. From Remark \ref{LORC.NOTOPEN}, we deduce that $\aleph_{\widetilde X}(\mathcal C^n) < \infty$ for every $n \geq 1$. Therefore, we can also define
        \begin{align}
            h_{top}(f; \mathcal C | \mathcal B) &:= \lim_{n \to \infty} \frac{\log \aleph_{\widetilde X} (\mathcal C^n | \mathcal B^n)}{n}. \label{COND.COND}
        \end{align}
        Finally, the {\em topological entropy of $f$ conditioned to $\mathcal B$} is given by:
        \begin{align}
            h_{top}(f | \mathcal B) &:= \sup\big\{ h_{top}(f; \mathcal C | \mathcal B)\ :\ \text{$\mathcal C$ is an $X$-open cover}\big\}. \label{COND.ENTROPY}
        \end{align}
        By an argument similar to the one given right after Definition \ref{ENTROPY.DEF}, as $X$ is compact, it is possible to take this supremum over all finite $X$-open covers. The following result will be used to prove a key inequality (see \eqref{FRANCISCA}), which will be used to conclude Theorem \ref{main3}.
        
        \begin{lemma}\label{COVER.B.N.}
            $\aleph_{\widetilde X}((\mathcal B \vee \mathcal C) | \mathcal E) \leq \aleph_{\widetilde X}(\mathcal B|\mathcal E) \cdot \aleph_{\widetilde X}(\mathcal C | (\mathcal B \vee \mathcal E))$ for every $\mathcal B, \mathcal C, \mathcal E$ covers of $\widetilde X$ such that $\aleph_{\widetilde X}(\mathcal C) < \infty$ and $\max\{\#\mathcal B, \#\mathcal E\} < \infty$. 
        \end{lemma}
        \begin{proof}
            Let $E \in \mathcal E$. Then, there exists a sub-collection $\mathcal M \subset \mathcal B$ such that
            \[
                E \subset \bigcup_{M \in \mathcal M} M \qquad \text{and} \qquad \#\mathcal M \leq \aleph_{\widetilde X}(\mathcal B|\mathcal E).
            \]
            Now, note that for every $M \in \mathcal M$, there exists a finite collection $\mathcal C_M \subset \mathcal C$ such that
            \[
                E \cap M \subset \bigcup_{C \in \mathcal C_M} C \qquad \text{and} \qquad \#\mathcal C_M \leq \aleph_{\widetilde X}(\mathcal C|(\mathcal B \vee \mathcal E)).
            \]
            Now, observe that $\mathcal F := \{M \cap C\ : \ M \in \mathcal M, \; C \in \mathcal C_M\}$ is a finite subfamily of $\mathcal B \vee \mathcal C$ satisfying $\aleph_E (\mathcal B \vee \mathcal C) \leq \# \mathcal F \leq \aleph_{\widetilde X}(\mathcal B | \mathcal E) \cdot \aleph_{\widetilde X}(\mathcal C | (\mathcal B \vee \mathcal E))$, which implies the result.
        \end{proof}
        
        \begin{remark}\label{REMARK.TECH}
            Taking $\mathcal E = \{X\}$, from part (a) of Lemma \ref{NPROPERTIES} and Lemma \ref{COVER.B.N.}, we deduce that
            \begin{equation}
                \aleph_{\widetilde X}(\mathcal C^n) \leq \aleph_{\widetilde X}(\mathcal B^n \vee \mathcal C^n) \leq \aleph_{\widetilde X}(\mathcal B^n) \cdot \aleph_{\widetilde X}(\mathcal C^n|\mathcal B^n), \label{NNN}
            \end{equation}
            for every $n \geq 1$, and $\mathcal B, \mathcal C$ covers of $\widetilde X$ such that $\#\mathcal B < \infty$ and $\aleph_{\widetilde X}(\mathcal C) < \infty$. Moreover, taking the logarithm and dividing by $n$ in \eqref{NNN}, as well as taking the supremum over all $X$-open covers $\mathcal C$, it follows that
            \begin{equation}
                h_{top}(f) \leq h_{top}(f, \mathcal B) + h_{top}(f | \mathcal B). \label{FRANCISCA}
            \end{equation}
        \end{remark}

        To prove the following result, we recall that $\mathcal D$ is the {\emph natural $X$-open cover} given by $\mathcal D = \{X_1, \ldots, X_N\}$.
        
        \begin{lemma}\label{TECH.C.E}
            $h_{top}(f, \mathcal C|\mathcal D) = 0$ for every finite $X$-open cover $\mathcal C$.
        \end{lemma}
        \begin{proof}
            Fix $n \geq 1$ and $D \in \mathcal D^n$. For $k \in \{1, \ldots , n\}$, and $C \in \mathcal C$, we define
            \begin{equation}
                D \cap f^{- k}\big(I(C)\big) = (f|_D)^{-k} \big(I(C)\big), \label{JOSEFA}
            \end{equation}
            where $I(C):= \big(\inf C, \sup C\big)$. From Lemma \ref{INTCAP}, the set defined by \eqref{JOSEFA} is either empty or an open interval contained in a connected component of $X \setminus \Delta^n$. Then, one of the following four cases occurs:\\[-2ex]
            \begin{itemize}
                \item $\inf f^{-k}\big(I(C)\big), \sup f^{-k}\big(I(C)\big) \in D$;
                \item $\inf f^{-k}\big(I(C)\big) < \inf D$ and $\sup f^{-k}\big(I(C)\big) \in D$;
                \item $\inf f^{-k}\big(I(C)\big) \in D$ and $\sup D < \sup f^{-k}\big(I(C)\big)$;
                \item $\inf f^{-k}\big(I(C)\big), \sup f^{-k}\big(I(C)\big) \notin D$.
            \end{itemize}
            
            \vspace*{1ex}
            In this way, $D$ contains at most two end-points of $I\big(f^{-k}\big(I(C)\big)\big)$. Therefore, we deduce that $D$ has at most $2n(\#\mathcal C)$ different end-points of the open intervals in 
            \[
                \mathcal E_n:= \big\{I\big(f^{-k}\big(I(C)\big)\big)\ :\ C \in \mathcal C, \, \; 0 \leq k \leq n - 1\big\}, \qquad \text{where} \quad I\big(f^{0}\big(I(C)\big)\big) = I(C).
            \]
            Since the elements of \(\mathcal C^n\) correspond to intersections of pre-images of elements in \(\mathcal C\), we can use the endpoints of intervals in \(\mathcal E_n\) to bound the number of elements in \(\mathcal C^n\) that intersect \(D\). In particular, as there are 4 possibilities in which an interval can intersect $D$, and each element in \(\mathcal C^n\) is contained in an interval whose endpoints can be chosen from the (at most) $2 n (\#\mathcal C)$ elements belonging to \(D\); that is,
            \begin{equation}
                \#\big\{E\in\mathcal C^n\ :\ E\cap D\neq \emptyset\big\}\leq 4 (2 n \# \mathcal C)^2. \label{INEQ.D1}
            \end{equation}
            Thus, from \eqref{INEQ.D1} we deduce 
            \begin{align*}
                h_{top}(f; \mathcal C|\mathcal D) &= \lim_{n \to \infty} \frac{\log \aleph_{\widetilde X} (\mathcal C^n|\mathcal D^n)}{n} \leq \lim_{n \to \infty} \frac{\log\big(16 n^2 (\#\mathcal C)^2\big)}{n} = 0.
            \end{align*}
        \end{proof}

        With this, to close this sub-section we prove a reformulation of the Misiurewicz-Szlenk formula established in Theorem \ref{main2}.
    
        \begin{theorem} \label{HTOP.LEQ}
            The following equality holds:   
            \[
                h_{top}(f) = h_{top}(f, \mathcal D) = \lim_{n \to \infty} \frac{\log c_n}{n}.
            \]
        \end{theorem}
        \begin{proof}
            From Lemma \ref{TECH.C.E} and the observation given right after \eqref{COND.ENTROPY}, it follows that $h_{top}(f|\mathcal D) = 0$. Furthermore, from Remark \ref{REMARK.TECH}, we deduce that
            \begin{equation}
                h_{top}(f) \leq h_{top}(f, \mathcal D) + h_{top}(f|\mathcal D) = h_{top}(f, \mathcal D).\label{INQ.FIN}
            \end{equation}
            Now, from \eqref{DANA} and Lemma \ref{CEPSILON}, we conclude that
            \[
               \lim_{n \to \infty} \frac{\log c_n}{n}\leq h_{top}(f) \leq \lim_{n \to \infty} \frac{\log \aleph_{\widetilde X}(\mathcal D^n)}{n} = \lim_{n \to \infty} \frac{\log c_n}{n},
            \]
            which concludes the proof.
        \end{proof}

        Furthermore, we can give a straightforward proof of Corollary \ref{GIORNO-GIOVANNA} by using Theorem \ref{HTOP.LEQ}. Recall that --in our context-- a pc-map is {\it injective} when it is injective over the union of its continuity pieces.
        \begin{corollary}\label{COSA}
            Let $f: X \to X$ be an injective pc-map. Then, $h_{top}(f) = 0$.
        \end{corollary}
        \begin{proof}
            First, as $f$ is injective, then $f$ has no critical continuity points (see the beginning of Section \ref{sec7} to recall the context).
        
            Now we prove that, as $f$ is injective, then $c_n \leq N + (n - 1) (N - 1)$ for every $n \geq 1$. To do this, we proceed by induction on $n$. First, if $n = 1$, the result follows. Now, assume that for some $k \geq 1$, we have
            \begin{align*}
                c_k \leq N + (k - 1) (N - 1).
            \end{align*}
            This implies that $f^k$ has at most $N + (k - 1) (N - 1)$ continuity pieces. Now, note that the discontinuities of $f^{k + 1}$ comprise two kinds of points. On the one hand, the extreme points of continuity pieces of $f^k$ and, on the other hand, the values of $x$ such that $f^k(x) \in \Delta$. However, as $f$ is injective, then $f^k(X \setminus \Delta^k)$ can intersect $\Delta$ at most $N - 1$ times. Therefore,
            \begin{align*}
                c_{k + 1} \leq c_k + N - 1 \leq N + k (N - 1),
            \end{align*}
            which concludes the first part of the proof. Now, thanks to Theorem \ref{main2}, we have that
            \begin{align*}
                h_{top}(f) = \lim_{n \to \infty} \frac{1}{n} \, \log c_n \leq \lim_{n \to \infty} \frac{1}{n} \, \log (N + (n - 1) (N - 1)) = 0,
            \end{align*}
            which lets us conclude the result.
        \end{proof}

    \subsection{Entropy of pc-maps with surjective branches} \label{GATALINA}
        We conclude Section \ref{sec7} by proving the last main result of this article. Let us recall that, throughout this section, we have globally assumed that $f$ is strictly monotonic on each of its continuity pieces. For this part, we assume that ``$f^{- j}(\Delta) \cap \Delta = \emptyset$ for all $j \geq 1$''. Although this condition depends on how $f$ is defined over $\Delta$, it is only an auxiliary assumption used to accurately count the number of preimages of $\Delta$. Nevertheless, as the definition of the map over $\Delta$ does not affect its topological entropy, then this condition can be easily induced by redefining the map over said set (see Remark \ref{INDEP.HTOP}).


        \begin{lemma}\label{TECH}
            If $\overline{f(X_i)} = X$ for each $i\in\{1, \ldots, N\}$ and $f^{-j}(\Delta)\cap\Delta = \emptyset$ for every $j \geq 1$, then $\#(\Delta^n) = N^n - 1$ for every $n \geq 0$.
        \end{lemma}
        \begin{proof} 
            We proceed by induction on $n$. Taking $n = 0$ we have that $\#(\Delta^0) = \#\emptyset= N^0 - 1 = 0$. Next, we assume that $\#(\Delta^k) = N^k - 1$ for some $k \geq 0$. Note that
            \[
                f^{-k}(\Delta) \cap \Delta^k = \bigcup_{j = 0}^{k - 1} f^{- j}(\Delta) \cap f^{- k}(\Delta) = \bigcup_{j = 0}^{k - 1} f^{- j}\big(\Delta \cap f^{- k + j}(\Delta)\big) = \emptyset,
            \]
            since, by hypothesis, the pre-images of $\Delta$ do not intersect $\Delta$. It follows that 
            \begin{equation}
                \#(\Delta^{k + 1}) = \#(f^{- k}(\Delta)) + \#(\Delta^k). \label{DELTAk}
            \end{equation}
            On the other hand, using the fact that $f$ is surjective on each branch, we deduce
            \begin{equation}
                \#(f^{- k}(\Delta)) = N \cdot \# f^{- (k - 1)}(\Delta) = N^k \cdot \#\Delta = N^k (N - 1). \label{PREIMAGEk}
            \end{equation}
            Finally, from the induction hypothesis and equalities \eqref{DELTAk} and \eqref{PREIMAGEk}, we conclude that
            \[
                \#(\Delta^{k + 1}) = N^k (N - 1) + N^k - 1 = N^{k + 1} - 1,
            \]
            which ends the proof.
        \end{proof}

        \begin{lemma}\label{LEQ.UNO}
            For every $n \geq 1$, $\big|\aleph_{\widetilde X}(\mathcal D^n) - \#(\Delta^n)\big| \leq 1$.
        \end{lemma}
        \begin{proof}
           Fix $n \geq 1$ and assume that $X = [a, b]$, where $a<b$. From Lemma \ref{PARTIAL.FND}, we have four cases. In one case, $a, b\notin \Delta^n$, which implies that $\#(\mathcal D^n) = \#(\Delta^n) + 1$, whilst the other cases arise from the possible alternatives to the former condition. In all of them, the inequality holds.
        \end{proof}

        Finally, we are in a position to provide a proof for the last main result of this article. For this, we reformulate Theorem \ref{main3} recalling that $N = c_1$ due to what we assumed at the beginning of this section.

        \begin{theorem}\label{FINAL.THEOREM}
            If $\overline{f(X_i)} = X$ for each $i \in \{1, \ldots, N\}$, then $h_{top}(f) = \log N$.
        \end{theorem}
        
        \begin{proof}
            Without loss of generality, we assume that $f^{- j}(\Delta) \cap \Delta = \emptyset$ for all $j \geq 1$.
            From Theorem \ref{HTOP.LEQ}, and Lemmas \ref{TECH} and \ref{LEQ.UNO}, it follows that
            \[
                \lim_{n \to \infty} \frac{\log\big(N^n - 2\big)}{n} \leq h_{top}(f) = \lim_{n \to \infty} \frac{\log \aleph_{\widetilde X}(\mathcal D^n)}{n} \leq \lim_{n \to \infty} \frac{\log(N^n)}{n},
            \]
            which concludes the proof.
        \end{proof}

\section{Some examples and applications} \label{sec8}

    We conclude this article by showing some examples and applications of the theory we have presented.

    \begin{example} \label{MAGDA}
        Consider pc-maps with two continuity pieces, whose representations are defined by the graphs shown in Figure \ref{fig:JULIANA}. As can be observed, pc-maps with two strictly monotonic continuity pieces such as the (discontinuous) doubling map, tent map, and even one-dimensional Lorenz maps that are surjective on each branch have a topological entropy equal to $\log 2$ (see Theorem \ref{main3}). We highlight that the pieces of $X$ do not need to have the same length and the map does not need to be only expanding or contracting in those intervals. Moreover, these examples can be generalized to maps with any number of continuity pieces as long as their number of critical points is finite (see Example \ref{MOD.N}).
        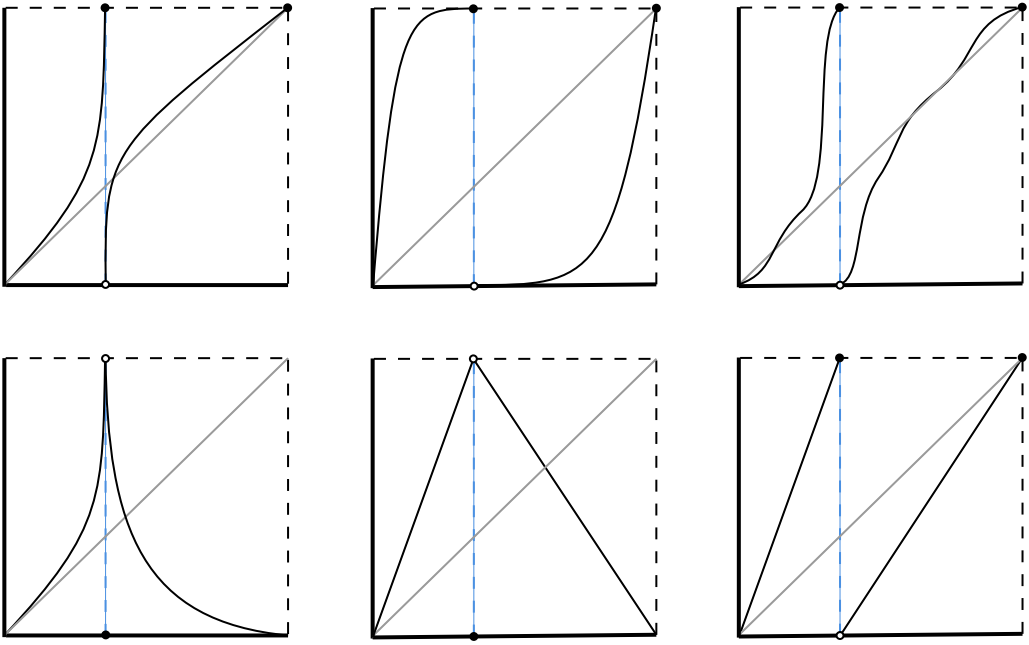
\begin{figure}
            \tikzset{every picture/.style={line width=0.75pt}} 

            \begin{tikzpicture}[x=0.75pt,y=0.75pt,yscale=-1,xscale=1,scale=.7]

                \draw [color={rgb, 255:red, 0; green, 0; blue, 0 }  ,draw opacity=1 ][fill={rgb, 255:red, 155; green, 155; blue, 155 }  ,fill opacity=1 ]   (362.48,303.06) -- (492.9,501.34) ;
                \draw [color={rgb, 255:red, 0; green, 0; blue, 0 }  ,draw opacity=1 ][fill={rgb, 255:red, 155; green, 155; blue, 155 }  ,fill opacity=1 ]   (623.54,302.75) -- (551.75,502.75) ;
                \draw [color={rgb, 255:red, 0; green, 0; blue, 0 }  ,draw opacity=1 ][fill={rgb, 255:red, 155; green, 155; blue, 155 }  ,fill opacity=1 ]   (753.65,302.54) -- (623.88,502.06) ;
                \draw    (623.95,249.49) .. controls (639.8,243.4) and (633.5,199.25) .. (651,173.25) .. controls (668.5,147.25) and (663,134.25) .. (693,110.5) .. controls (723,86.75) and (714,61) .. (753.59,50.28) ;
                \draw [color={rgb, 255:red, 0; green, 0; blue, 0 }  ,draw opacity=1 ][line width=1.5]    (290.69,251.47) -- (492.9,249.54) ;
                \draw [color={rgb, 255:red, 0; green, 0; blue, 0 }  ,draw opacity=1 ][line width=1.5]    (29.43,250.04) -- (230.4,250.04) ;
                \draw [color={rgb, 255:red, 0; green, 0; blue, 0 }  ,draw opacity=0.36 ][line width=0.75]  [dash pattern={on 4.5pt off 4.5pt}]  (230.4,249.04) -- (230.4,47.76) ;
                \draw [color={rgb, 255:red, 0; green, 0; blue, 0 }  ,draw opacity=0.36 ][line width=0.75]  [dash pattern={on 4.5pt off 4.5pt}]  (29.12,50.76) -- (230.09,50.76) ;
                \draw [color={rgb, 255:red, 74; green, 144; blue, 226 }  ,draw opacity=0.48 ][fill={rgb, 255:red, 74; green, 144; blue, 226 }  ,fill opacity=1 ][line width=0.75]  [dash pattern={on 4.5pt off 4.5pt}]  (100.32,250.28) -- (100.32,49.01) ;
                \draw    (28.5,249.66) .. controls (100,172.25) and (98.25,150) .. (99.98,50.76) ;
                \draw [color={rgb, 255:red, 155; green, 155; blue, 155 }  ,draw opacity=0.64 ][fill={rgb, 255:red, 155; green, 155; blue, 155 }  ,fill opacity=1 ]   (230.4,50.76) -- (28.5,249.66) ;
                \draw    (230.09,50.76) .. controls (100.26,152.66) and (99.48,152.66) .. (100.45,249.66) ;
                \draw  [fill={rgb, 255:red, 0; green, 0; blue, 0 }  ,fill opacity=1 ] (97.51,50.76) .. controls (97.51,49.39) and (98.61,48.28) .. (99.98,48.28) .. controls (101.35,48.28) and (102.46,49.39) .. (102.46,50.76) .. controls (102.46,52.13) and (101.35,53.24) .. (99.98,53.24) .. controls (98.61,53.24) and (97.51,52.13) .. (97.51,50.76) -- cycle ;
                \draw  [fill={rgb, 255:red, 255; green, 255; blue, 255 }  ,fill opacity=1 ] (98.77,249.66) .. controls (98.77,248.73) and (99.52,247.97) .. (100.45,247.97) .. controls (101.39,247.97) and (102.14,248.73) .. (102.14,249.66) .. controls (102.14,250.59) and (101.39,251.35) .. (100.45,251.35) .. controls (99.52,251.35) and (98.77,250.59) .. (98.77,249.66) -- cycle ;
                \draw [color={rgb, 255:red, 0; green, 0; blue, 0 }  ,draw opacity=0.36 ][line width=0.75]  [dash pattern={on 4.5pt off 4.5pt}]  (492.9,249.54) -- (492.9,48.26) ;
                \draw [color={rgb, 255:red, 0; green, 0; blue, 0 }  ,draw opacity=0.36 ][line width=0.75]  [dash pattern={on 4.5pt off 4.5pt}]  (291.62,51.26) -- (492.59,51.26) ;
                \draw [color={rgb, 255:red, 74; green, 144; blue, 226 }  ,draw opacity=0.48 ][fill={rgb, 255:red, 74; green, 144; blue, 226 }  ,fill opacity=1 ][line width=0.75]  [dash pattern={on 4.5pt off 4.5pt}]  (362.82,250.78) -- (362.82,49.51) ;
                \draw    (291,250.16) .. controls (307,52.5) and (316,51.5) .. (362.48,51.26) ;
                \draw [color={rgb, 255:red, 155; green, 155; blue, 155 }  ,draw opacity=0.64 ][fill={rgb, 255:red, 155; green, 155; blue, 155 }  ,fill opacity=1 ]   (492.9,51.26) -- (291,250.16) ;
                \draw  [fill={rgb, 255:red, 0; green, 0; blue, 0 }  ,fill opacity=1 ] (360.79,51.26) .. controls (360.79,50.33) and (361.55,49.57) .. (362.48,49.57) .. controls (363.42,49.57) and (364.17,50.33) .. (364.17,51.26) .. controls (364.17,52.19) and (363.42,52.95) .. (362.48,52.95) .. controls (361.55,52.95) and (360.79,52.19) .. (360.79,51.26) -- cycle ;
                \draw    (492.59,51.26) .. controls (463.69,248.92) and (447.44,249.92) .. (363.5,250.16) ;
                \draw  [fill={rgb, 255:red, 255; green, 255; blue, 255 }  ,fill opacity=1 ] (361.13,250.78) .. controls (361.13,249.85) and (361.89,249.09) .. (362.82,249.09) .. controls (363.75,249.09) and (364.51,249.85) .. (364.51,250.78) .. controls (364.51,251.71) and (363.75,252.47) .. (362.82,252.47) .. controls (361.89,252.47) and (361.13,251.71) .. (361.13,250.78) -- cycle ;
                \draw [color={rgb, 255:red, 0; green, 0; blue, 0 }  ,draw opacity=1 ][line width=1.5]    (28.19,251.25) -- (28.19,50.75) ;
                \draw [color={rgb, 255:red, 0; green, 0; blue, 0 }  ,draw opacity=1 ][line width=1.5]    (290.69,252.25) -- (290.69,50.97) ;
                \draw [color={rgb, 255:red, 0; green, 0; blue, 0 }  ,draw opacity=1 ][line width=1.5]    (551.69,250.8) -- (753.9,248.87) ;
                \draw [color={rgb, 255:red, 0; green, 0; blue, 0 }  ,draw opacity=0.36 ][line width=0.75]  [dash pattern={on 4.5pt off 4.5pt}]  (753.9,248.87) -- (753.9,47.6) ;
                \draw [color={rgb, 255:red, 0; green, 0; blue, 0 }  ,draw opacity=0.36 ][line width=0.75]  [dash pattern={on 4.5pt off 4.5pt}]  (552.62,50.59) -- (753.59,50.59) ;
                \draw [color={rgb, 255:red, 74; green, 144; blue, 226 }  ,draw opacity=0.48 ][fill={rgb, 255:red, 74; green, 144; blue, 226 }  ,fill opacity=1 ][line width=0.75]  [dash pattern={on 4.5pt off 4.5pt}]  (623.82,250.12) -- (623.82,48.84) ;
                \draw [color={rgb, 255:red, 155; green, 155; blue, 155 }  ,draw opacity=0.64 ][fill={rgb, 255:red, 155; green, 155; blue, 155 }  ,fill opacity=1 ]   (753.9,50.6) -- (552,249.49) ;
                \draw  [fill={rgb, 255:red, 0; green, 0; blue, 0 }  ,fill opacity=1 ] (621.79,50.59) .. controls (621.79,49.66) and (622.55,48.9) .. (623.48,48.9) .. controls (624.42,48.9) and (625.17,49.66) .. (625.17,50.59) .. controls (625.17,51.52) and (624.42,52.28) .. (623.48,52.28) .. controls (622.55,52.28) and (621.79,51.52) .. (621.79,50.59) -- cycle ;
                \draw  [fill={rgb, 255:red, 255; green, 255; blue, 255 }  ,fill opacity=1 ] (622.27,249.49) .. controls (622.27,248.56) and (623.02,247.8) .. (623.95,247.8) .. controls (624.89,247.8) and (625.64,248.56) .. (625.64,249.49) .. controls (625.64,250.42) and (624.89,251.18) .. (623.95,251.18) .. controls (623.02,251.18) and (622.27,250.42) .. (622.27,249.49) -- cycle ;
                \draw [color={rgb, 255:red, 0; green, 0; blue, 0 }  ,draw opacity=1 ][line width=1.5]    (551.69,251.58) -- (551.69,50.3) ;
                \draw    (552,249.49) .. controls (578.5,240.24) and (572.42,219.44) .. (596.71,196.72) .. controls (621,174) and (603.25,72.99) .. (623.48,50.59) ;
                \draw  [fill={rgb, 255:red, 255; green, 255; blue, 255 }  ,fill opacity=1 ] (98.63,249.59) .. controls (98.63,248.66) and (99.39,247.91) .. (100.32,247.91) .. controls (101.25,247.91) and (102.01,248.66) .. (102.01,249.59) .. controls (102.01,250.53) and (101.25,251.28) .. (100.32,251.28) .. controls (99.39,251.28) and (98.63,250.53) .. (98.63,249.59) -- cycle ;
                \draw  [fill={rgb, 255:red, 0; green, 0; blue, 0 }  ,fill opacity=1 ] (228.71,50.76) .. controls (228.71,49.83) and (229.47,49.08) .. (230.4,49.08) .. controls (231.33,49.08) and (232.09,49.83) .. (232.09,50.76) .. controls (232.09,51.7) and (231.33,52.45) .. (230.4,52.45) .. controls (229.47,52.45) and (228.71,51.7) .. (228.71,50.76) -- cycle ;
                \draw  [fill={rgb, 255:red, 0; green, 0; blue, 0 }  ,fill opacity=1 ] (491.21,51.26) .. controls (491.21,50.33) and (491.97,49.58) .. (492.9,49.58) .. controls (493.83,49.58) and (494.59,50.33) .. (494.59,51.26) .. controls (494.59,52.2) and (493.83,52.95) .. (492.9,52.95) .. controls (491.97,52.95) and (491.21,52.2) .. (491.21,51.26) -- cycle ;
                \draw  [fill={rgb, 255:red, 0; green, 0; blue, 0 }  ,fill opacity=1 ] (752.21,50.29) .. controls (752.21,49.35) and (752.97,48.6) .. (753.9,48.6) .. controls (754.83,48.6) and (755.59,49.35) .. (755.59,50.29) .. controls (755.59,51.22) and (754.83,51.97) .. (753.9,51.97) .. controls (752.97,51.97) and (752.21,51.22) .. (752.21,50.29) -- cycle ;
                \draw  [fill={rgb, 255:red, 255; green, 255; blue, 255 }  ,fill opacity=1 ] (97.84,249.59) .. controls (97.84,248.23) and (98.95,247.12) .. (100.32,247.12) .. controls (101.69,247.12) and (102.8,248.23) .. (102.8,249.59) .. controls (102.8,250.96) and (101.69,252.07) .. (100.32,252.07) .. controls (98.95,252.07) and (97.84,250.96) .. (97.84,249.59) -- cycle ;
                \draw  [fill={rgb, 255:red, 255; green, 255; blue, 255 }  ,fill opacity=1 ] (360.55,250.78) .. controls (360.55,249.41) and (361.66,248.3) .. (363.03,248.3) .. controls (364.4,248.3) and (365.51,249.41) .. (365.51,250.78) .. controls (365.51,252.15) and (364.4,253.26) .. (363.03,253.26) .. controls (361.66,253.26) and (360.55,252.15) .. (360.55,250.78) -- cycle ;
                \draw  [fill={rgb, 255:red, 0; green, 0; blue, 0 }  ,fill opacity=1 ] (227.61,50.76) .. controls (227.61,49.39) and (228.72,48.28) .. (230.09,48.28) .. controls (231.46,48.28) and (232.57,49.39) .. (232.57,50.76) .. controls (232.57,52.13) and (231.46,53.24) .. (230.09,53.24) .. controls (228.72,53.24) and (227.61,52.13) .. (227.61,50.76) -- cycle ;
                \draw  [fill={rgb, 255:red, 0; green, 0; blue, 0 }  ,fill opacity=1 ] (360.01,51.47) .. controls (360.01,50.1) and (361.11,48.99) .. (362.48,48.99) .. controls (363.85,48.99) and (364.96,50.1) .. (364.96,51.47) .. controls (364.96,52.84) and (363.85,53.95) .. (362.48,53.95) .. controls (361.11,53.95) and (360.01,52.84) .. (360.01,51.47) -- cycle ;
                \draw  [fill={rgb, 255:red, 0; green, 0; blue, 0 }  ,fill opacity=1 ] (490.42,51.05) .. controls (490.42,49.68) and (491.53,48.58) .. (492.9,48.58) .. controls (494.27,48.58) and (495.38,49.68) .. (495.38,51.05) .. controls (495.38,52.42) and (494.27,53.53) .. (492.9,53.53) .. controls (491.53,53.53) and (490.42,52.42) .. (490.42,51.05) -- cycle ;
                \draw  [fill={rgb, 255:red, 0; green, 0; blue, 0 }  ,fill opacity=1 ] (621.01,50.59) .. controls (621.01,49.22) and (622.11,48.11) .. (623.48,48.11) .. controls (624.85,48.11) and (625.96,49.22) .. (625.96,50.59) .. controls (625.96,51.96) and (624.85,53.07) .. (623.48,53.07) .. controls (622.11,53.07) and (621.01,51.96) .. (621.01,50.59) -- cycle ;
                \draw  [fill={rgb, 255:red, 0; green, 0; blue, 0 }  ,fill opacity=1 ] (751.21,50.29) .. controls (751.21,48.92) and (752.32,47.81) .. (753.69,47.81) .. controls (755.06,47.81) and (756.17,48.92) .. (756.17,50.29) .. controls (756.17,51.65) and (755.06,52.76) .. (753.69,52.76) .. controls (752.32,52.76) and (751.21,51.65) .. (751.21,50.29) -- cycle ;
                \draw  [fill={rgb, 255:red, 255; green, 255; blue, 255 }  ,fill opacity=1 ] (621.34,250.12) .. controls (621.34,248.75) and (622.45,247.64) .. (623.82,247.64) .. controls (625.19,247.64) and (626.3,248.75) .. (626.3,250.12) .. controls (626.3,251.48) and (625.19,252.59) .. (623.82,252.59) .. controls (622.45,252.59) and (621.34,251.48) .. (621.34,250.12) -- cycle ;
                \draw [color={rgb, 255:red, 0; green, 0; blue, 0 }  ,draw opacity=1 ][line width=1.5]    (290.69,503.27) -- (492.9,501.34) ;
                \draw [color={rgb, 255:red, 0; green, 0; blue, 0 }  ,draw opacity=1 ][line width=1.5]    (29.43,501.84) -- (230.4,501.84) ;
                \draw [color={rgb, 255:red, 0; green, 0; blue, 0 }  ,draw opacity=0.36 ][line width=0.75]  [dash pattern={on 4.5pt off 4.5pt}]  (230.4,500.84) -- (230.4,299.56) ;
                \draw [color={rgb, 255:red, 0; green, 0; blue, 0 }  ,draw opacity=0.36 ][line width=0.75]  [dash pattern={on 4.5pt off 4.5pt}]  (29.12,302.56) -- (230.09,302.56) ;
                \draw [color={rgb, 255:red, 74; green, 144; blue, 226 }  ,draw opacity=0.48 ][fill={rgb, 255:red, 74; green, 144; blue, 226 }  ,fill opacity=1 ][line width=0.75]  [dash pattern={on 4.5pt off 4.5pt}]  (100.32,502.08) -- (100.32,300.81) ;
                \draw    (28.5,501.46) .. controls (100,424.05) and (98.25,401.8) .. (99.98,302.56) ;
                \draw [color={rgb, 255:red, 155; green, 155; blue, 155 }  ,draw opacity=0.64 ][fill={rgb, 255:red, 155; green, 155; blue, 155 }  ,fill opacity=1 ]   (230.4,302.56) -- (28.5,501.46) ;
                \draw  [fill={rgb, 255:red, 255; green, 255; blue, 255 }  ,fill opacity=1 ] (98.77,501.46) .. controls (98.77,500.53) and (99.52,499.77) .. (100.45,499.77) .. controls (101.39,499.77) and (102.14,500.53) .. (102.14,501.46) .. controls (102.14,502.39) and (101.39,503.15) .. (100.45,503.15) .. controls (99.52,503.15) and (98.77,502.39) .. (98.77,501.46) -- cycle ;
                \draw [color={rgb, 255:red, 0; green, 0; blue, 0 }  ,draw opacity=0.36 ][line width=0.75]  [dash pattern={on 4.5pt off 4.5pt}]  (492.9,501.34) -- (492.9,300.06) ;
                \draw [color={rgb, 255:red, 0; green, 0; blue, 0 }  ,draw opacity=0.36 ][line width=0.75]  [dash pattern={on 4.5pt off 4.5pt}]  (291.62,303.06) -- (492.59,303.06) ;
                \draw [color={rgb, 255:red, 74; green, 144; blue, 226 }  ,draw opacity=0.48 ][fill={rgb, 255:red, 74; green, 144; blue, 226 }  ,fill opacity=1 ][line width=0.75]  [dash pattern={on 4.5pt off 4.5pt}]  (362.82,502.58) -- (362.82,301.31) ;
                \draw [color={rgb, 255:red, 155; green, 155; blue, 155 }  ,draw opacity=0.64 ][fill={rgb, 255:red, 155; green, 155; blue, 155 }  ,fill opacity=1 ]   (492.9,303.06) -- (291,501.96) ;
                \draw  [fill={rgb, 255:red, 0; green, 0; blue, 0 }  ,fill opacity=1 ] (360.79,303.06) .. controls (360.79,302.13) and (361.55,301.37) .. (362.48,301.37) .. controls (363.42,301.37) and (364.17,302.13) .. (364.17,303.06) .. controls (364.17,303.99) and (363.42,304.75) .. (362.48,304.75) .. controls (361.55,304.75) and (360.79,303.99) .. (360.79,303.06) -- cycle ;
                \draw  [fill={rgb, 255:red, 255; green, 255; blue, 255 }  ,fill opacity=1 ] (361.13,502.58) .. controls (361.13,501.65) and (361.89,500.89) .. (362.82,500.89) .. controls (363.75,500.89) and (364.51,501.65) .. (364.51,502.58) .. controls (364.51,503.51) and (363.75,504.27) .. (362.82,504.27) .. controls (361.89,504.27) and (361.13,503.51) .. (361.13,502.58) -- cycle ;
                \draw [color={rgb, 255:red, 0; green, 0; blue, 0 }  ,draw opacity=1 ][line width=1.5]    (28.19,503.05) -- (28.19,302.55) ;
                \draw [color={rgb, 255:red, 0; green, 0; blue, 0 }  ,draw opacity=1 ][line width=1.5]    (290.69,504.05) -- (290.69,302.77) ;
                \draw [color={rgb, 255:red, 0; green, 0; blue, 0 }  ,draw opacity=1 ][line width=1.5]    (551.69,502.6) -- (753.9,500.67) ;
                \draw [color={rgb, 255:red, 0; green, 0; blue, 0 }  ,draw opacity=0.36 ][line width=0.75]  [dash pattern={on 4.5pt off 4.5pt}]  (753.9,500.67) -- (753.9,299.4) ;
                \draw [color={rgb, 255:red, 0; green, 0; blue, 0 }  ,draw opacity=0.36 ][line width=0.75]  [dash pattern={on 4.5pt off 4.5pt}]  (552.62,302.39) -- (753.59,302.39) ;
                \draw [color={rgb, 255:red, 74; green, 144; blue, 226 }  ,draw opacity=0.48 ][fill={rgb, 255:red, 74; green, 144; blue, 226 }  ,fill opacity=1 ][line width=0.75]  [dash pattern={on 4.5pt off 4.5pt}]  (623.82,501.92) -- (623.82,300.64) ;
                \draw [color={rgb, 255:red, 155; green, 155; blue, 155 }  ,draw opacity=0.64 ][fill={rgb, 255:red, 155; green, 155; blue, 155 }  ,fill opacity=1 ]   (753.9,302.4) -- (552,501.29) ;
                \draw  [fill={rgb, 255:red, 0; green, 0; blue, 0 }  ,fill opacity=1 ] (621.79,302.39) .. controls (621.79,301.46) and (622.55,300.7) .. (623.48,300.7) .. controls (624.42,300.7) and (625.17,301.46) .. (625.17,302.39) .. controls (625.17,303.32) and (624.42,304.08) .. (623.48,304.08) .. controls (622.55,304.08) and (621.79,303.32) .. (621.79,302.39) -- cycle ;
                \draw  [fill={rgb, 255:red, 255; green, 255; blue, 255 }  ,fill opacity=1 ] (622.27,501.29) .. controls (622.27,500.36) and (623.02,499.6) .. (623.95,499.6) .. controls (624.89,499.6) and (625.64,500.36) .. (625.64,501.29) .. controls (625.64,502.22) and (624.89,502.98) .. (623.95,502.98) .. controls (623.02,502.98) and (622.27,502.22) .. (622.27,501.29) -- cycle ;
                \draw [color={rgb, 255:red, 0; green, 0; blue, 0 }  ,draw opacity=1 ][line width=1.5]    (551.69,503.38) -- (551.69,302.1) ;
                \draw  [fill={rgb, 255:red, 0; green, 0; blue, 0 }  ,fill opacity=1 ] (752.21,302.09) .. controls (752.21,301.15) and (752.97,300.4) .. (753.9,300.4) .. controls (754.83,300.4) and (755.59,301.15) .. (755.59,302.09) .. controls (755.59,303.02) and (754.83,303.77) .. (753.9,303.77) .. controls (752.97,303.77) and (752.21,303.02) .. (752.21,302.09) -- cycle ;
                \draw  [fill={rgb, 255:red, 0; green, 0; blue, 0 }  ,fill opacity=1 ] (360.34,502.58) .. controls (360.34,501.21) and (361.45,500.1) .. (362.82,500.1) .. controls (364.19,500.1) and (365.3,501.21) .. (365.3,502.58) .. controls (365.3,503.95) and (364.19,505.06) .. (362.82,505.06) .. controls (361.45,505.06) and (360.34,503.95) .. (360.34,502.58) -- cycle ;
                \draw  [fill={rgb, 255:red, 0; green, 0; blue, 0 }  ,fill opacity=1 ] (621.01,302.39) .. controls (621.01,301.02) and (622.11,299.91) .. (623.48,299.91) .. controls (624.85,299.91) and (625.96,301.02) .. (625.96,302.39) .. controls (625.96,303.76) and (624.85,304.87) .. (623.48,304.87) .. controls (622.11,304.87) and (621.01,303.76) .. (621.01,302.39) -- cycle ;
                \draw  [fill={rgb, 255:red, 0; green, 0; blue, 0 }  ,fill opacity=1 ] (751.21,302.09) .. controls (751.21,300.72) and (752.32,299.61) .. (753.69,299.61) .. controls (755.06,299.61) and (756.17,300.72) .. (756.17,302.09) .. controls (756.17,303.45) and (755.06,304.56) .. (753.69,304.56) .. controls (752.32,304.56) and (751.21,303.45) .. (751.21,302.09) -- cycle ;
                \draw  [fill={rgb, 255:red, 255; green, 255; blue, 255 }  ,fill opacity=1 ] (621.34,501.92) .. controls (621.34,500.55) and (622.45,499.44) .. (623.82,499.44) .. controls (625.19,499.44) and (626.3,500.55) .. (626.3,501.92) .. controls (626.3,503.28) and (625.19,504.39) .. (623.82,504.39) .. controls (622.45,504.39) and (621.34,503.28) .. (621.34,501.92) -- cycle ;
                \draw    (100.32,300.81) .. controls (101.4,427.4) and (127.4,491.8) .. (230.4,501.84) ;
                \draw  [fill={rgb, 255:red, 255; green, 255; blue, 255 }  ,fill opacity=1 ] (97.84,302.81) .. controls (97.84,301.44) and (98.95,300.33) .. (100.32,300.33) .. controls (101.69,300.33) and (102.8,301.44) .. (102.8,302.81) .. controls (102.8,304.18) and (101.69,305.29) .. (100.32,305.29) .. controls (98.95,305.29) and (97.84,304.18) .. (97.84,302.81) -- cycle ;
                \draw  [fill={rgb, 255:red, 0; green, 0; blue, 0 }  ,fill opacity=1 ] (97.98,501.46) .. controls (97.98,500.09) and (99.09,498.98) .. (100.45,498.98) .. controls (101.82,498.98) and (102.93,500.09) .. (102.93,501.46) .. controls (102.93,502.83) and (101.82,503.94) .. (100.45,503.94) .. controls (99.09,503.94) and (97.98,502.83) .. (97.98,501.46) -- cycle ;
                \draw [color={rgb, 255:red, 0; green, 0; blue, 0 }  ,draw opacity=1 ][fill={rgb, 255:red, 155; green, 155; blue, 155 }  ,fill opacity=1 ]   (362.48,303.27) -- (290.69,503.27) ;
                \draw  [fill={rgb, 255:red, 255; green, 255; blue, 255 }  ,fill opacity=1 ] (360.01,303.06) .. controls (360.01,301.69) and (361.11,300.58) .. (362.48,300.58) .. controls (363.85,300.58) and (364.96,301.69) .. (364.96,303.06) .. controls (364.96,304.43) and (363.85,305.54) .. (362.48,305.54) .. controls (361.11,305.54) and (360.01,304.43) .. (360.01,303.06) -- cycle ;
            \end{tikzpicture}
            \centering
            \caption{Examples of pc-maps with two strictly monotonic continuity pieces with a topological entropy of $\log 2$.}
            \label{fig:JULIANA}
        \end{figure}
    \end{example}
    
    \begin{example}\label{MOD.N}
        Let $N \geq 2$. Using Theorem \ref{main3}, it can be shown that the pc-map $T_N: [0, 1] \to [0, 1]$ given by
        \[
            T_N(x) = \left\{\!\!\begin{array}{ll}
                N x, & \text{ if $x \in [0, 1/N)$,}
                \\[.5ex]
                N x - j, & \text{ if $x \in \big[j/N, (j + 1)/N\big)$ and $1 \leq j \leq N - 2$,}
                \\[.5ex]
                N x - N + 1, & \text{ if $x \in \big[(N - 1)/N, 1\big]$,}
            \end{array}\right.
        \]
        has a topological entropy equal to $\log N$. In particular, $T_2$ corresponds to the discontinuous doubling map defined on $[0, 1]$.
    \end{example}

    \begin{example}\label{ANZIE}
        In Example \ref{MAGDA}, we highlighted the class of maps with more than one continuity piece that are surjective in each of them. Due to Lemma \ref{QUEQUE}, we have that some maps that are not surjective in all of their continuity pieces can also have positive entropy. In Figure \ref{fig:GABRIELA}, we show a map with four strictly monotonic continuity pieces with two pseudo-invariant sets: $X$, and the set comprising the last two continuity pieces. Furthermore, as the tent map in the last two branches is surjective in each of them, we can ensure that the topological entropy of this map will be greater than or equal to $\log 2$.
        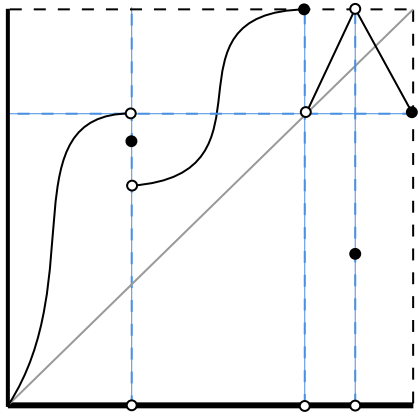
\begin{figure}
            \tikzset{every picture/.style={line width=0.75pt}} 

            \begin{tikzpicture}[x=0.75pt,y=0.75pt,yscale=-1,xscale=1,scale=1]

                \draw    (201.48,58) .. controls (124.4,59.33) and (192,141) .. (114.67,146.67) ;
                \draw [color={rgb, 255:red, 74; green, 144; blue, 226 }  ,draw opacity=0.48 ][fill={rgb, 255:red, 74; green, 144; blue, 226 }  ,fill opacity=1 ][line width=0.75]  [dash pattern={on 4.5pt off 4.5pt}]  (114.48,258.19) -- (114.48,56.92) ;
                \draw [color={rgb, 255:red, 0; green, 0; blue, 0 }  ,draw opacity=1 ][line width=2.25]    (52.69,257.21) -- (255,257.21) ;
                \draw [color={rgb, 255:red, 74; green, 144; blue, 226 }  ,draw opacity=0.48 ][fill={rgb, 255:red, 74; green, 144; blue, 226 }  ,fill opacity=1 ][line width=0.75]  [dash pattern={on 4.5pt off 4.5pt}]  (255.5,110.5) -- (52.5,110.5) ;
                \draw [color={rgb, 255:red, 0; green, 0; blue, 0 }  ,draw opacity=1 ][fill={rgb, 255:red, 155; green, 155; blue, 155 }  ,fill opacity=1 ]   (226,57.75) -- (201.33,110.67) ;
                \draw [color={rgb, 255:red, 0; green, 0; blue, 0 }  ,draw opacity=0.36 ][line width=0.75]  [dash pattern={on 4.5pt off 4.5pt}]  (254.9,256.28) -- (254.9,55) ;
                \draw [color={rgb, 255:red, 0; green, 0; blue, 0 }  ,draw opacity=0.36 ][line width=0.75]  [dash pattern={on 4.5pt off 4.5pt}]  (53.62,58) -- (254.59,58) ;
                \draw [color={rgb, 255:red, 74; green, 144; blue, 226 }  ,draw opacity=0.48 ][fill={rgb, 255:red, 74; green, 144; blue, 226 }  ,fill opacity=1 ][line width=0.75]  [dash pattern={on 4.5pt off 4.5pt}]  (200.82,257.52) -- (200.82,56.25) ;
                \draw [color={rgb, 255:red, 155; green, 155; blue, 155 }  ,draw opacity=0.64 ][fill={rgb, 255:red, 155; green, 155; blue, 155 }  ,fill opacity=1 ]   (254.9,58) -- (53,256.9) ;
                \draw [color={rgb, 255:red, 0; green, 0; blue, 0 }  ,draw opacity=1 ][line width=1.5]    (52.69,257.99) -- (52.69,57.71) ;
                \draw  [fill={rgb, 255:red, 0; green, 0; blue, 0 }  ,fill opacity=1 ] (198.01,58) .. controls (198.01,56.63) and (199.11,55.52) .. (200.48,55.52) .. controls (201.85,55.52) and (202.96,56.63) .. (202.96,58) .. controls (202.96,59.37) and (201.85,60.48) .. (200.48,60.48) .. controls (199.11,60.48) and (198.01,59.37) .. (198.01,58) -- cycle ;
                \draw  [fill={rgb, 255:red, 0; green, 0; blue, 0 }  ,fill opacity=1 ] (251.77,109.75) .. controls (251.77,108.38) and (252.88,107.27) .. (254.25,107.27) .. controls (255.62,107.27) and (256.73,108.38) .. (256.73,109.75) .. controls (256.73,111.12) and (255.62,112.23) .. (254.25,112.23) .. controls (252.88,112.23) and (251.77,111.12) .. (251.77,109.75) -- cycle ;
                \draw  [fill={rgb, 255:red, 255; green, 255; blue, 255 }  ,fill opacity=1 ] (198.34,257.52) .. controls (198.34,256.15) and (199.45,255.04) .. (200.82,255.04) .. controls (202.19,255.04) and (203.3,256.15) .. (203.3,257.52) .. controls (203.3,258.89) and (202.19,260) .. (200.82,260) .. controls (199.45,260) and (198.34,258.89) .. (198.34,257.52) -- cycle ;
                \draw [color={rgb, 255:red, 0; green, 0; blue, 0 }  ,draw opacity=1 ][fill={rgb, 255:red, 155; green, 155; blue, 155 }  ,fill opacity=1 ]   (254.25,109.75) -- (226,57.75) ;
                \draw [color={rgb, 255:red, 74; green, 144; blue, 226 }  ,draw opacity=0.48 ][fill={rgb, 255:red, 74; green, 144; blue, 226 }  ,fill opacity=1 ][line width=0.75]  [dash pattern={on 4.5pt off 4.5pt}]  (226,257.33) -- (226,57.75) ;
                \draw  [fill={rgb, 255:red, 255; green, 255; blue, 255 }  ,fill opacity=1 ] (223.52,257.33) .. controls (223.52,255.96) and (224.63,254.86) .. (226,254.86) .. controls (227.37,254.86) and (228.48,255.96) .. (228.48,257.33) .. controls (228.48,258.7) and (227.37,259.81) .. (226,259.81) .. controls (224.63,259.81) and (223.52,258.7) .. (223.52,257.33) -- cycle ;
                \draw    (53,256.9) .. controls (93.67,193.67) and (53.33,110.67) .. (114,110.33) ;
                \draw  [fill={rgb, 255:red, 0; green, 0; blue, 0 }  ,fill opacity=1 ] (111.86,124.33) .. controls (111.86,122.96) and (112.96,121.86) .. (114.33,121.86) .. controls (115.7,121.86) and (116.81,122.96) .. (116.81,124.33) .. controls (116.81,125.7) and (115.7,126.81) .. (114.33,126.81) .. controls (112.96,126.81) and (111.86,125.7) .. (111.86,124.33) -- cycle ;
                \draw  [fill={rgb, 255:red, 255; green, 255; blue, 255 }  ,fill opacity=1 ] (198.86,109.67) .. controls (198.86,108.3) and (199.96,107.19) .. (201.33,107.19) .. controls (202.7,107.19) and (203.81,108.3) .. (203.81,109.67) .. controls (203.81,111.04) and (202.7,112.14) .. (201.33,112.14) .. controls (199.96,112.14) and (198.86,111.04) .. (198.86,109.67) -- cycle ;
                \draw  [fill={rgb, 255:red, 255; green, 255; blue, 255 }  ,fill opacity=1 ] (111.52,110.33) .. controls (111.52,108.96) and (112.63,107.86) .. (114,107.86) .. controls (115.37,107.86) and (116.48,108.96) .. (116.48,110.33) .. controls (116.48,111.7) and (115.37,112.81) .. (114,112.81) .. controls (112.63,112.81) and (111.52,111.7) .. (111.52,110.33) -- cycle ;
                \draw  [fill={rgb, 255:red, 255; green, 255; blue, 255 }  ,fill opacity=1 ] (112.19,146.67) .. controls (112.19,145.3) and (113.3,144.19) .. (114.67,144.19) .. controls (116.04,144.19) and (117.14,145.3) .. (117.14,146.67) .. controls (117.14,148.04) and (116.04,149.14) .. (114.67,149.14) .. controls (113.3,149.14) and (112.19,148.04) .. (112.19,146.67) -- cycle ;
                \draw  [fill={rgb, 255:red, 255; green, 255; blue, 255 }  ,fill opacity=1 ] (112.01,257.19) .. controls (112.01,255.82) and (113.12,254.71) .. (114.48,254.71) .. controls (115.85,254.71) and (116.96,255.82) .. (116.96,257.19) .. controls (116.96,258.56) and (115.85,259.67) .. (114.48,259.67) .. controls (113.12,259.67) and (112.01,258.56) .. (112.01,257.19) -- cycle ;
                \draw  [fill={rgb, 255:red, 0; green, 0; blue, 0 }  ,fill opacity=1 ] (223.52,181) .. controls (223.52,179.63) and (224.63,178.52) .. (226,178.52) .. controls (227.37,178.52) and (228.48,179.63) .. (228.48,181) .. controls (228.48,182.37) and (227.37,183.48) .. (226,183.48) .. controls (224.63,183.48) and (223.52,182.37) .. (223.52,181) -- cycle ;
                \draw  [fill={rgb, 255:red, 255; green, 255; blue, 255 }  ,fill opacity=1 ] (223.52,57.75) .. controls (223.52,56.38) and (224.63,55.27) .. (226,55.27) .. controls (227.37,55.27) and (228.48,56.38) .. (228.48,57.75) .. controls (228.48,59.12) and (227.37,60.23) .. (226,60.23) .. controls (224.63,60.23) and (223.52,59.12) .. (223.52,57.75) -- cycle ;
            \end{tikzpicture}
            \centering
            \caption{Graph of a pc-map with four continuity pieces with a positive entropy greater than or equal to $\log 2$.}
            \label{fig:GABRIELA}
        \end{figure}
    \end{example}

    \begin{example}[{[Entropy of interval exchange transformations]}] \label{IET}
        The {\it generalized interval exchange transformations} (GIET) defined in \cite{NN} are a type of bijective pc-map with zero topological entropy (see Corollary \ref{GIORNO-GIOVANNA}). From this fact, together with Lemma \ref{CONJ}, we deduce that any GIET cannot be topologically conjugate to a map like those shown in Examples \ref{MAGDA} -- \ref{ANZIE}.
    \end{example}

    \begin{example}{{[Entropy of injective piecewise contracting interval maps]}}
        In \cite{GN22}, the authors show that an injective piecewise contracting interval map, which is increasing on each of its continuity pieces, has \underline{singular entropy} equal to zero; that is, 
        \[
            \lim_{n \to \infty} \frac{1}{n} \log c_n = 0.
        \]
        Therefore, from Corollary \ref{GIORNO-GIOVANNA}, we infer that the topological entropy of these maps equals zero even when not all their branches are increasing, as long as these maps are injective.
        \\[1ex]
        Independently, the authors in \cite{CV23} prove --using arguments related to atoms and symbolic complexity-- that injective piecewise contracting interval maps restricted to $\widetilde X$ have topological entropy equal to zero when computed via Bowen's formula.  Moreover, from Theorem \ref{BOWEN.MAIN} and Corollary \ref{GIORNO-GIOVANNA}, we infer that their topological entropy equals zero and it is independent of the metric used to compute it.
    \end{example}

    \noindent{\bf Closing remark:} Naturally, the next step in this study is the \textit{variational principle}, which is out of the scope of this article. B. Pires addressed a delicate point along this line in \cite{P16}, where he proves the existence of regular Borel probability measures that are invariant for pc-maps {\it without connections}, which are the maps $f$ such that the one-sided limits of $f(x)$ as $x$ approaches a point in $\Delta$ belong to $\widetilde X$ (see rigorous definition in \cite{P16}). Thus, the challenge of establishing a variational principle in this context will primarily rely on adapting the classical proof using our notion of topological entropy for pc-maps, assuming the ``no connections'' hypothesis.\\[-1ex]

\noindent{\bf Acknowledgments:} A.E.C. was supported by ANID Fondecyt Research Initiation N$^{\circ}$11230064, ANID Fondecyt Regular N$^{\circ}$1230569 and MathAmsud 22-MATH-10. E.V-S. was funded by ANID, Beca Chile Doctorado en el extranjero, number 72210071.

\end{document}